\documentclass[11pt]{article}
\usepackage{amsmath, amsfonts, amsthm, amssymb}
\usepackage{fullpage}

\usepackage[x11names, rgb]{xcolor}
\usepackage{graphicx}
\usepackage{tikz}
\usetikzlibrary{decorations,arrows,shapes}

\usepackage[colorlinks=true,linkcolor=blue!50!black,citecolor=blue!50!black]{hyperref}
\usepackage{etoolbox}
\usepackage{enumerate}
\usepackage{listings}
\usepackage{booktabs}
\usepackage{comment}
\usepackage[affil-it]{authblk}
\usepackage{float}
\usepackage{algorithm2e}
\usepackage{mathtools}

\newtheorem{theorem}{Theorem}[section]
\newtheorem{lemma}[theorem]{Lemma}

\newtheorem{corollary}[theorem]{Corollary}

\newtheorem{remark}[theorem]{Remark}

\newcommand{\T}{\mathbb{T}}
\newcommand{\R}{\mathbb{R}}
\newcommand{\Z}{\mathbb{Z}}
\newcommand{\E}{\mathbb{E}}
\newcommand{\Var}{\operatorname{Var}}

\newcommand{\dd}{\,\mathrm{d}}

\newcommand{\1}{\mathbf{1}}

\author{Shi Feng\thanks{Cornell University. Email: \texttt{sf599@cornell.edu}}}
\date{\today}
\title{Fluctuations of the quadratic matching cost on the flat torus:\\ dimensions five and higher}
\begin{document}
\maketitle
\begin{abstract}
We prove a central limit theorem for the quadratic
matching cost between the uniform measure and its
empirical counterpart on the flat torus in dimensions
$d\ge5$. Let $\mu$ be Haar probability measure on
$\T^d$ and let $\mu_n$ be the empirical measure of
$n$ independent samples from $\mu$. Then
\begin{align*}
\frac{W_2^2(\mu_n,\mu)-\E W_2^2(\mu_n,\mu)}
{\sqrt{\Var(W_2^2(\mu_n,\mu))}}
&\Longrightarrow N(0,1),\\
c_dn^{-1-4/d}
\le \Var(W_2^2(\mu_n,\mu))
&\le C_dn^{-1-4/d},
\end{align*}
for positive constants $c_d,C_d$ depending only on $d$ and all sufficiently
large $n$.

The proof introduces a random-walk approach to
matching fluctuations. We construct a random walk
that jumps between neighboring cells of the optimal
partition and use it to describe how the matching
cost responds to changes in the sample. Estimates
for this walk control the effect of replacing one
sample point and the interaction between two
replacements. These bounds yield the variance
estimates and allow us to apply a martingale
central limit theorem.
\end{abstract}

\setcounter{tocdepth}{2}
\tableofcontents

\section{Introduction}

Optimal transport compares probability distributions
by finding the least costly way to move mass from
one to the other, and connects problems in probability,
analysis, geometry, and statistics \cite{Villani09}. When one distribution is observed
through a random sample, the resulting transport cost
is itself random. Understanding its fluctuations is
a basic question in statistical optimal transport;
see \cite{dBGSLRV25} for a survey of distributional
results and their applications. In this paper, we
study quadratic matching between the uniform measure
and its empirical counterpart on the flat torus.

Random matching has a long history in geometric
probability and random optimization
\cite{AKT84,Tal94a,Tal94b,Yuk98}.
Much of this work concerns the typical size of the
optimal cost. A law of large numbers for the minimum
total bipartite matching length in dimensions
$d\ge3$ was proved in \cite{DY95}.
More recently, \cite{AGGT26} proved convergence
of the rescaled expected bipartite matching
length on the two-dimensional flat torus
to a positive finite constant.

For quadratic matching, the dependence of the
mean cost on dimension is now well understood.
Following the linearization proposed in
\cite{CLPS14}, the PDE approach developed in
\cite{AST19} identified the leading constant
in dimension two. In dimensions at least three,
\cite{GT21} proved convergence of the rescaled
expected costs for matching on a cube, both
between two independent samples and between
a sample and the uniform measure.
These results describe the leading cost of
matching, but leave open the finer question
of how that cost varies between independent samples.

The PDE approach also reveals structure in the
optimal allocation itself. In dimension two,
Ambrosio, Glaudo, and Trevisan
\cite{AGT19} obtained a quantitative
$L^2$ approximation of the optimal transport
map using a Poisson equation of the
heat-smoothed empirical density. Goldman, Huesmann,
and Otto \cite{GHO21} developed quantitative
linearization estimates for the Monge--Amp\`ere
equation with rough data. Building on this theory,
Goldman and Huesmann \cite{GH} proved that suitably
rescaled and spatially averaged optimal
displacements converge to a curl-free Gaussian
free field in dimensions two and three, with an
additional normalization in dimension two.
Their estimates provide geometric information
that we also use here.

Central limit theorems for empirical transport
costs are also available when the sampling
distribution differs from the reference measure.
Writing $\rho_n$ for the empirical measure of
$n$ independent samples from $\rho$, these
results take the form
\begin{align*}
\sqrt n\left(
W_2^2(\rho_n,\nu)-\E W_2^2(\rho_n,\nu)
\right)
\Longrightarrow N(0,\sigma_{\rho,\nu}^2),
\end{align*}
under suitable assumptions on $\rho$ and $\nu$.
This was established for quadratic cost in
\cite{dBL19}, extended to more general costs
in \cite{dBGSL24}, and proved on the flat torus
in \cite[Theorem~2.6]{GDGSCN23}.
In our setting, the sample is drawn from the
reference measure itself, so the Gaussian limit
above has zero variance. These central limit
theorems therefore do not identify the appropriate
fluctuation scale or a nontrivial limiting law.
The problem of identifying such limits was
highlighted in
\cite[Section~7, Problem~3]{dBGSLRV25}.

Limit theorems for coinciding sampling and
reference measures have been established
in one dimension under suitable assumptions
\cite{dBGU05} and for measures supported on
fixed finite spaces \cite{SM18}. For normalized volume on
closed Riemannian manifolds of dimension at most
three, Ledoux \cite{Led19} proved a limit theorem
for the renormalized dual Sobolev energy and
deduced fluctuation results for the quadratic
transport cost of suitably heat-smoothed
empirical measures. For the empirical measure
without smoothing, the preceding work
\cite{FM26} determines the variance asymptotics
and limiting distributions on the flat torus
in dimensions two, three, and four.
The limits are non-Gaussian in dimensions two
and three. Dimension four is critical: the
limit is Gaussian, but the variance has an
additional logarithmic factor.

The present paper treats every dimension $d\ge5$.
We prove that the variance of the squared
Wasserstein distance has order $n^{-1-4/d}$ and
that the centered cost, divided by its standard
deviation, converges to a standard Gaussian
distribution. Together with \cite{FM26}, this
determines the variance order and standardized
limiting law for quadratic matching on the flat
torus in every dimension $d\ge2$.

Our proof relates fluctuations of the matching
cost to a random walk on the cells of the optimal
partition. To study the replacement of one sample
point, we gradually transfer its assigned mass
to the replacement point while keeping all
locations fixed. The response is governed by
the derivatives of the cell masses with respect
to the weights defining the cells.
These derivatives form a weighted graph
Laplacian, a structure used in Newton methods
for semi-discrete optimal transport
\cite{KMT19}. The corresponding quadratic
form also appears in the analysis of vanishing
entropic regularization \cite{ANS22}.
We use the associated random walk to control
the resulting change in cost and the interaction
between two replacements. These estimates yield
the variance bounds and allow us to apply a
martingale central limit theorem. This differs
from the Fourier and Hoeffding analysis
used in \cite{FM26}.

\subsection{Model and results}\label{sec:model}

Let $\T^d=(\R/\Z)^d$ be the flat unit torus,
equipped with the distance
\begin{align*}
\rho_1(x,y):=\min_{k\in\Z^d}|x-y+k|.
\end{align*}
Let $\mu$ be the uniform probability measure
on $\T^d$, and let $Y_1,\ldots,Y_n$ be independent
with distribution $\mu$. Their empirical measure is
\begin{align*}
\mu_n:=\frac1n\sum_{i=1}^n\delta_{Y_i}.
\end{align*}
We study the quadratic matching cost
\begin{align}\label{eq:matching-cost}
W_n:=W_2^2(\mu_n,\mu)
=\inf_{\pi\in\Pi(\mu_n,\mu)}
\int_{\T^d\times\T^d}\rho_1(x,y)^2\dd\pi(x,y),
\end{align}
where $\Pi(\mu_n,\mu)$ is the set of probability
measures on $\T^d\times\T^d$ with marginals
$\mu_n$ and $\mu$.
Equivalently, we partition the torus into $n$
regions of volume $1/n$, assign one region to
each $Y_i$, and minimize the sum of the integrals
of $\rho_1(y,Y_i)^2$ over the assigned regions.

\begin{theorem}\label{thm:clt}
Let $d\ge5$. There exist constants
$0<c_d\le C_d<\infty$ such that, for all
sufficiently large $n$,
\begin{align}\label{eq:variance-order}
c_dn^{-1-4/d}
\le\Var(W_n)
\le C_dn^{-1-4/d}.
\end{align}
Moreover, as $n\to\infty$,
\begin{align}\label{eq:variance-normalized-clt}
\frac{W_n-\E W_n}{\sqrt{\Var(W_n)}}
\Longrightarrow N(0,1).
\end{align}
\end{theorem}

Combining Theorem~\ref{thm:clt} with
\cite[Theorems~1.1--1.2]{FM26} gives the
following description of the variance and
limiting distribution in every dimension
$d\ge2$.

\begin{theorem}\label{thm:all-dimensions}
For every $d\ge2$, the variance and limiting
distribution of $W_n$ satisfy the following
as $n\to\infty$:
\begin{center}
\renewcommand{\arraystretch}{2}
\begin{tabular}{c c c}
\hline
Dimension & Variance & Limiting distribution\\
\hline
$d=2,3$
& $\Var(W_n)\sim v_dn^{-2}$
& $n(W_n-\E W_n)\Longrightarrow Z_d$\\
$d=4$
& $\displaystyle
   \Var(W_n)\sim\frac{\log n}{16\pi^2n^2}$
& $\displaystyle
   \frac{n(W_n-\E W_n)}{\sqrt{\log n}}
   \Longrightarrow N\left(0,\frac1{16\pi^2}\right)$\\
$d\ge5$
& $\Var(W_n)\asymp_d n^{-1-4/d}$
& $\displaystyle
   \frac{W_n-\E W_n}{\sqrt{\Var(W_n)}}
   \Longrightarrow N(0,1)$\\
\hline
\end{tabular}
\end{center}
For $d=2,3$,
\begin{align*}
v_d
&:=\frac1{8\pi^4}
\sum_{k\in\Z^d\setminus\{0\}}|k|^{-4},\\
Z_d
&:=\sum_{[k]\in(\Z^d\setminus\{0\})/\{\pm1\}}
\frac{\zeta_{[k]}-1}{2\pi^2|k|^2},
\end{align*}
where the $\zeta_{[k]}$ are independent
$\operatorname{Exp}(1)$ random variables. Here $\asymp_d$ means bounded above and below
by positive constants depending only on $d$.
\end{theorem}

\subsection{Proof strategy}\label{sec:proof-strategy}

In this section, we explain the main ideas behind
the proof through two questions:
\begin{enumerate}
\item How does studying the fluctuations of the
quadratic matching cost lead to a random walk
on the optimal partition?
\item Why does this method work in dimensions
$d\ge5$?
\end{enumerate}

Our proof uses a martingale central limit theorem.
We first rescale the torus so that each sample
point receives mass one. Put $L=n^{1/d}$ and
$X_i:=LY_i$. On $\T_L^d=(\R/L\Z)^d$, write
\begin{align*}
\rho_L(x,y)&:=\min_{k\in\Z^d}|x-y+Lk|,\\
\mu_L(\dd y)&:=\frac{\dd y}{n},
\end{align*}
where $\dd y$ is Lebesgue measure of total mass
$n$. The points $X_i$ are independent with
common law $\mu_L$.

Define the total quadratic matching cost by
\begin{align*}
C_L:=W_2^2\left(\sum_{i=1}^n\delta_{X_i},\dd y\right),
\end{align*}
where $W_2^2$ is defined for finite measures
of equal mass without normalizing their masses.
Rescaling space and mass gives
\begin{align}\label{eq:cost-wasserstein-scaling}
C_L=nL^2W_n=n^{1+2/d}W_n.
\end{align}
It therefore suffices to prove that $\Var(C_L)$
has order $n$ and that its standardized
fluctuations converge to a standard Gaussian
distribution.

We reveal the sample points one at a time.
Let $\mathcal F_i:=\sigma(X_1,\ldots,X_i)$,
with $\mathcal F_0$ trivial, and let $C_L^{(i)}$
be the cost after replacing $X_i$ by an independent
uniform point, independent of the original sample.
Then
\begin{align}\label{eq:strategy-martingale}
C_L-\E C_L&=\sum_{i=1}^n\xi_i,
\end{align}
where
\begin{align}\label{eq:strategy-replacement}
\xi_i
&:=\E[C_L\mid\mathcal F_i]
-\E[C_L\mid\mathcal F_{i-1}]\notag\\
&=\E[C_L-C_L^{(i)}\mid\mathcal F_i].
\end{align}
This identity connects the martingale differences
to replacement costs. We first bound the effect
of replacing one sample point, which gives the
variance upper bound through the Efron--Stein
inequality. To prove the central limit theorem,
we further control how this replacement cost
changes when a second point is replaced.
Together with moment bounds and a separate
variance lower bound, these estimates allow us
to apply the martingale central limit theorem.

To study one replacement, fix the other points
$z_1,\ldots,z_{n-1}$ and let $p$ and $q$ be the
original and replacement points. We gradually
transfer the unit mass at $p$ to $q$:
\begin{align}\label{eq:strategy-interpolation}
\nu_s
&:=\sum_{i=1}^{n-1}\delta_{z_i}
+(1-s)\delta_p+s\delta_q,\notag\\
F(s)&:=W_2^2(\nu_s,\dd y),
\qquad 0\le s\le1.
\end{align}
Here at stage $s$,
$p$ receives mass $1-s$, $q$ receives mass $s$,
and each remaining point receives mass one.
Thus $F(0)$ is the original cost and $F(1)$ is
the cost after replacement. When $z_1,\ldots,z_{n-1},p,q$ are independent
and uniform on $\T_L^d$, the Efron--Stein
inequality gives
\begin{align}\label{eq:strategy-efron-stein}
\Var(C_L)
\le\frac n2\E|F(1)-F(0)|^2.
\end{align}
Thus, to prove the variance upper bound,
it suffices to bound this expectation by some constant $C_d$.

Write $\mathcal Z=\{z_1,\ldots,z_{n-1},p,q\}$.
Kantorovich duality gives
\begin{align}\label{eq:strategy-dual}
F(s)=\max_{w\in\R^{\mathcal Z}}
\Bigg\{
&\sum_{i=1}^{n-1}w_{z_i}+(1-s)w_p+sw_q\notag\\
&-\int_{\T_L^d}
\max_{r\in\mathcal Z}
\{w_r-\rho_L(y,r)^2\}\dd y
\Bigg\}.
\end{align}
Here $w_r$ is the weight attached to the point $r$.
Define the associated potential by
\begin{align*}
\phi_w(y):=\max_{u\in\mathcal Z}
\{w_u-\rho_L(y,u)^2\}.
\end{align*}
For an optimal weight vector, each $y\in\T_L^d$
is assigned to a point attaining this maximum.
The cell assigned to $r$ is
\begin{align*}
C_r(w):=\left\{y\in\T_L^d:
w_r-\rho_L(y,r)^2=\phi_w(y)
\right\}.
\end{align*}
The optimal weights give these cells their
prescribed masses:
\begin{align*}
|C_{z_i}(w)|=1,\qquad
|C_p(w)|=1-s,\qquad
|C_q(w)|=s.
\end{align*}

Let $w(s)$ be an optimal weight vector for $\nu_s$,
and write $\phi_s:=\phi_{w(s)}$. Choose the
normalization so that
\begin{align*}
\int_{\T_L^d}\phi_s(y)\dd y=0.
\end{align*}
Differentiating the dual formula
\eqref{eq:strategy-dual} gives
$F'(s)=w_q(s)-w_p(s)$, and hence
\begin{align}\label{eq:strategy-replacement-integral}
F(1)-F(0)
=\int_0^1\bigl(w_q(s)-w_p(s)\bigr)\dd s;
\end{align}
see Lemma~\ref{lem:split-paths}.
The weights are difficult to control directly,
but their derivatives admit an explicit
representation, which we derive below.
We therefore integrate along the interpolation
to express $w(s)$ in terms of its endpoint
values and its derivative. A direct computation in Lemma~\ref{lem:replacement-response-bound}
gives
\begin{align*}
\left(\E|F(1)-F(0)|^2\right)^{1/2}
&\le
\left(\E|w_q(1)|^2\right)^{1/2}
+\left(\E|w_p(0)|^2\right)^{1/2}\\
&\quad+\int_0^1
\left(
\E\sum_{r\in\mathcal Z}m_r(s)
\left|\frac{\dd w_r(s)}{\dd s}\right|^2
\right)^{1/2}\dd s.
\end{align*}
The difficult part is to prove
\begin{align}\label{eq:strategy-integrated-response}
\int_0^1
\left(
\E\sum_{r\in\mathcal Z}m_r(s)
\left|\frac{\dd w_r(s)}{\dd s}\right|^2
\right)^{1/2}\dd s
\le C_d.
\end{align}
Once this estimate is available, the endpoint
weight moments follow from the cell-radius
estimates in Appendix~\ref{app:radius-moments}
and elementary arguments in the proof of
Corollary~\ref{cor:variance-upper-bound}.
Together, these bounds give
$\E|F(1)-F(0)|^2\le C_d$.

To estimate these derivatives, we first describe
how changing the weights changes the cell masses.
Write $m_r(w):=|C_r(w)|$ and define
\begin{align*}
H_{ru}(w):=\frac{\partial m_r(w)}{\partial w_u}.
\end{align*}
For intuition, we work locally in Euclidean
coordinates and suppress the periodicity of the
torus. Then the entries are given by
\begin{align}\label{eq:strategy-laplacian}
H_{ru}(w)
=
\begin{cases}
\displaystyle
-\frac{\mathcal H^{d-1}
(\partial C_r(w)\cap\partial C_u(w))}
{2|r-u|},
& r\ne u,\\[2ex]
\displaystyle
-\sum_{v\ne r}H_{rv}(w),
& r=u,
\end{cases}
\end{align}
where $\mathcal H^{d-1}$ denotes surface measure;
see Lemma~\ref{lem:mass-derivative} and
\cite{KMT19}.
Increasing a weight expands its cell and takes
mass from neighboring cells. The above definition makes $H$ the weighted Laplacian matrix of a connected
graph. Therefore, its kernel consists of constant vectors, and
it is invertible on the subspace of zero-sum
vectors. Denote this inverse by $H^+$.
For every $h$ with $\sum_r h_r=0$,
\begin{align*}
u=H^+h
\quad\Longleftrightarrow\quad
Hu=h,\qquad \sum_r u_r=0.
\end{align*}

Along our interpolation, the mass at $q$ increases
at rate one, the mass at $p$ decreases at rate
one, and all other masses remain unchanged.
The chain rule therefore gives
\begin{align}\label{eq:strategy-weight-response}
H(w(s))\frac{\dd w(s)}{\dd s}
=\frac{\dd}{\dd s}m(w(s))
=e_q-e_p,
\qquad 0<s<1,
\end{align}
where $e_r$ is the vector equal to one at $r$
and zero elsewhere. Applying $H^+$ determines the derivative up to
a common additive constant:
\begin{align}\label{eq:strategy-inverse-response}
\frac{\dd w(s)}{\dd s}
=H(w(s))^+(e_q-e_p)+c(s)\1.
\end{align}
Since $\partial_s\phi_s(y)=\dd w_r(s)/\dd s$
for almost every $y\in C_r(w(s))$, our
normalization fixes $c(s)$ through
\begin{align*}
0
=\frac{\dd}{\dd s}\int_{\T_L^d}\phi_s(y)\dd y
=\sum_{r\in\mathcal Z}m_r(s)
\frac{\dd w_r(s)}{\dd s}.
\end{align*}

The matrix $H$ has the explicit geometric formula
\eqref{eq:strategy-laplacian}, but its inverse
$H^+$ is harder to estimate directly.
We therefore use an integral representation
of the inverse on the appropriate subspace.
The matrix exponentials in this representation
describe the transition probabilities of a
random walk. This connects the weight derivatives
to the walk and allows us to use probabilistic
and analytic tools to estimate them.

Fix $s\in(0,1)$ and write $H=H(w(s))$,
$m_r=m_r(s)$, and $D:=\operatorname{Diag}(m)$.
The operator $D^{-1}H$ preserves the subspace
\begin{align*}
\left\{v:\sum_{r\in\mathcal Z}m_rv_r=0\right\},
\end{align*}
on which it is invertible and has positive
eigenvalues. Both $\dd w(s)/\dd s$ and
$D^{-1}(e_q-e_p)$ belong to this subspace.
Applying the integral representation of the
inverse on this subspace gives
\begin{align*}
\frac{\dd w(s)}{\dd s}
&=(D^{-1}H)^{-1}D^{-1}(e_q-e_p)\\
&=\int_0^\infty
e^{-tD^{-1}H}D^{-1}(e_q-e_p)\dd t.
\end{align*}

The matrix $-D^{-1}H$ is the generator of a
continuous-time random walk $(X_t)_{t\ge0}$
that jumps from $r$ to $u\ne r$ at rate
$-H_{ru}/m_r$. The walk is reversible with
stationary distribution $(m_r/n)_r$, and
\begin{align*}
\bigl(e^{-tD^{-1}H}\bigr)_{ru}
=\mathbb P_r(X_t=u),
\end{align*}
where $\mathbb P_r$ denotes the law of the walk
started at $r$, with the points in $\mathcal Z$ and their masses fixed.
Taking the $r$-th component of the preceding
representation therefore gives
\begin{align*}
\frac{\dd w_r(s)}{\dd s}
=\int_0^\infty
\left(
\frac{\mathbb P_r(X_t=q)}{m_q}
-\frac{\mathbb P_r(X_t=p)}{m_p}
\right)\dd t;
\end{align*}
see Lemma~\ref{lem:walk-response}.

This representation, together with the semigroup
property and reversibility, gives
\begin{align}\label{eq:strategy-response-return}
\sum_{r\in\mathcal Z}m_r
\left|\frac{\dd w_r(s)}{\dd s}\right|^2
\le
2\int_0^\infty t\left[
\frac{\mathbb P_p(X_t=p)}{m_p}
+\frac{\mathbb P_q(X_t=q)}{m_q}
-\frac2n
\right]\dd t.
\end{align}
as shown in \eqref{eq:response-return-bound}.
The subtraction removes the stationary
contributions. To prove \eqref{eq:strategy-integrated-response},
it therefore suffices to bound the return
probabilities of $X_t$.

This suggests the dimension threshold.
For a simple random walk on $\Z^d$, return
probabilities decay like $t^{-d/2}$, and
\begin{align*}
\int_1^\infty t\,t^{-d/2}\dd t<\infty
\quad\Longleftrightarrow\quad d\geq 5.
\end{align*}
The stability of heat-kernel estimates under
suitable changes in graph geometry and jump
rates \cite{Kumagai14} motivates this comparison.
We prove the required return-probability estimates
directly from the geometry of the optimal cells.
These estimates bound the right-hand side of
\eqref{eq:strategy-response-return}, giving
\eqref{eq:strategy-integrated-response}
and therefore $\Var(C_L)\le C_dn$.

To apply the martingale central limit theorem
\cite[Theorem~8.2.4]{Durrett19}, we verify the
Lindeberg condition \eqref{eq:conditional-lindeberg} using the fourth-moment
bound in Lemma~\ref{lem:replacement-fourth-moment}
and the variance lower bound in
Lemma~\ref{lem:variance-lower-bound}.
The main difficulty is to prove
\begin{align}\label{eq:strategy-conditional-variance}
\frac{\sum_{i=1}^n
\E[\xi_i^2\mid\mathcal F_{i-1}]}
{\Var(C_L)}
\longrightarrow1
\qquad\text{in probability}.
\end{align}
By Lemma~\ref{lem:conditional-variance} and the
variance lower bound, to prove
\eqref{eq:strategy-conditional-variance},
it suffices to obtain suitable bounds on the
interaction \eqref{eq:strategy-interaction}
between two replacements.

For two replacements,
fix the remaining points $z_1,\ldots,z_{n-2}$
and define
\begin{align*}
\nu_{s,t}
&:=\sum_{i=1}^{n-2}\delta_{z_i}
+(1-s)\delta_{p_1}+s\delta_{q_1}
+(1-t)\delta_{p_2}+t\delta_{q_2},\\
F(s,t)&:=W_2^2(\nu_{s,t},\dd y),
\qquad 0\le s,t\le1.
\end{align*}
The difference
\begin{align}\label{eq:strategy-interaction}
F(1,1)-F(1,0)-F(0,1)+F(0,0)
\end{align}
measures how the cost of replacing $p_1$ by $q_1$
changes when $p_2$ is also replaced by $q_2$.
Differentiating as in the one-replacement
argument connects this interaction to the
random walk. Estimates for its transition
probabilities give the required decay of the
interaction.
The martingale central limit theorem then
proves Theorem~\ref{thm:clt}.

The paper is organized as follows.
Section~\ref{sec:cells} develops the derivative
formulas relating the transport cost, cell masses,
and optimal weights.
Section~\ref{sec:geometry} introduces the associated
random walk and uses its return estimates to prove
the variance upper bound.
Section~\ref{sec:mixed-replacements} bounds the
interaction between two replacements and the
fourth moment of a single replacement cost.
Section~\ref{sec:clt} proves the variance lower
bound and the central limit theorem.
Appendix~\ref{app:radius-moments} establishes the
cell-radius estimates,
Appendix~\ref{sec:hoeffding} recalls the Hoeffding
decomposition, and
Appendix~\ref{app:direct-poincare} explains why
a direct application of the product Poincar\'e
inequality does not give the correct variance order.

\section{Cells, weights, and derivatives}\label{sec:cells}

In this section, we derive how the cell volumes and optimal
transport cost change when the masses or point locations vary.
The results hold in every dimension $d\ge2$.
The derivative of the cell volumes with respect to the
weights is standard in semi-discrete optimal transport;
see \cite[Theorem~1.3]{KMT19}.
For completeness, we derive this formula directly in the
periodic setting, together with the other derivative
formulas needed below.

\subsection{The deterministic allocation problem}
\label{sec:cells-weights}

Fix a finite set of distinct points
$\mathcal Z\subset\T_L^d$, which we call sites.
These are the points to which mass is assigned.
Write $\R^{\mathcal Z}$ for the space of real vectors
$u=(u_p)_{p\in\mathcal Z}$.
For $p\in\mathcal Z$, let $e_p$ be the coordinate vector
whose component at $p$ is one and whose other components are zero.

Let $m=(m_p)_{p\in\mathcal Z}$ be strictly positive masses satisfying
$\sum_{p\in\mathcal Z}m_p=n$.
The mass $m_p$ is the volume that must be assigned to $p$.
Define
\begin{align*}
 F_{\mathcal Z}(m)
 :=W_2^2\left(\sum_{p\in\mathcal Z}m_p\delta_p,dy\right).
\end{align*}
Thus $F_{\mathcal Z}(m)$ is the minimum total squared distance needed to
transport uniform volume to the sites with these prescribed masses.

By Kantorovich duality,
\begin{align*}
F_{\mathcal Z}(m)
=\max_{w\in\R^{\mathcal Z}}
\left\{
\sum_{p\in\mathcal Z}m_pw_p
-\int_{\T_L^d}\max_{q\in\mathcal Z}
\{w_q-\rho_L(y,q)^2\}\dd y
\right\}.
\end{align*}
For a weight vector $w\in\R^{\mathcal Z}$, define the potential, cells,
and cell-volume vector by
\begin{align*}
 \phi_w(y)&=\max_{q\in\mathcal Z}\{w_q-\rho_L(y,q)^2\},\\
 C_p(w)&=\{y\in\T_L^d:w_p-\rho_L(y,p)^2=\phi_w(y)\},
 \qquad p\in\mathcal Z,\\
 m(w)&=(|C_p(w)|)_{p\in\mathcal Z}.
\end{align*}
At any maximizing weight vector $w$, the cells
have the prescribed volumes:
\begin{align*}
|C_p(w)|=m_p,
\qquad p\in\mathcal Z.
\end{align*}
The cell $C_p(w)$ is the region transported to $p$ and has volume
$m_p$. These cells form an optimal partition up to sets of zero
volume. Increasing $w_p$ makes assignment to $p$ more favorable.

Adding the same constant to all weights changes neither
the cells nor the dual value. It adds that constant to
the potential. We therefore choose the unique normalization
\begin{align}\label{eq:normalization}
 \int_{\T_L^d}\phi_w(y)\dd y=0.
\end{align}
Indeed, since $|\T_L^d|=n$, this is obtained by subtracting
$n^{-1}\int_{\T_L^d}\phi_w(y)\dd y$ from every weight.
With this normalization, the dual formula becomes
\begin{align*}
 F_{\mathcal Z}(m)=\sum_{p\in\mathcal Z}m_pw_p.
\end{align*}

\subsection{Mass response}
\label{ssec:mass-response}

For distinct sites $p,q\in\mathcal Z$, denote the common interface by
$\Sigma_{pq}=\partial C_p\cap\partial C_q$.
To describe it, represent the torus by $[0,L]^d$.
Locally, we have
\begin{align*}
 \rho_L(y,p)=|y-p|,\qquad \rho_L(y,q)=|y-q|.
\end{align*}
On $\Sigma_{pq}$, the weighted costs agree:
\begin{align*}
 |y-p|^2-w_p=|y-q|^2-w_q.
\end{align*}
Cancelling the quadratic terms gives
\begin{align}\label{eq:face-equation}
 2y\cdot(q-p)=|q|^2-|p|^2+w_p-w_q.
\end{align}
Thus each face piece lies in a hyperplane with outward
unit normal $(q-p)/|q-p|$ from $C_p$. Therefore, the image to keep in mind is that the cells $C_p$ are polytopes and the vector $q-p$ is perpendicular to the interface $\Sigma_{pq}$.

The faces consist of finitely
many flat pieces. For each piece $F$, let $\ell_F=|q-p|$ and define the conductance
\begin{align*}
 c_{pq}
 =\sum_{F\subset\Sigma_{pq}}
   \frac{\mathcal H^{d-1}(F)}{2\ell_F},
 \qquad p\ne q.
\end{align*}
Here $\mathcal H^{d-1}$ denotes surface measure. Each
face is counted once. The boundary of the fundamental cube
introduces no additional face. Set $c_{pp}=0$ and write
$p\sim q$ when $c_{pq}>0$.

The coefficient $c_{pq}$ measures the rate at which volume is
transferred from $C_q$ to $C_p$ when $w_p$ increases and the
other weights remain fixed. Define the matrix
$H=(H_{pq})_{p,q\in\mathcal Z}$ by
\begin{align*}
 H_{pq}
 =
 \begin{cases}
 \displaystyle\sum_{r\in\mathcal Z\setminus\{p\}}c_{pr},&p=q,\\
 -c_{pq},&p\ne q.
 \end{cases}
\end{align*}
This is the weighted graph Laplacian of the cell adjacency graph.
Since $c_{pq}=c_{qp}$, it is symmetric and
\begin{align*}
 u^{\mathsf T}Hu
 =\frac12\sum_{p,q\in\mathcal Z}c_{pq}(u_p-u_q)^2.
\end{align*}
Note that
\begin{align*}
 u^{\mathsf T}Hu=0
 \quad\Longleftrightarrow\quad
 u\text{ is constant}.
\end{align*}
Since $H$ is symmetric and $H\1=0$, it maps the zero-sum
subspace into itself and is positive definite there.
Its restriction to this subspace is therefore invertible.

We denote this inverse by $H^+$ and extend it by zero on
constant vectors. Thus, for $\sum_{p\in\mathcal Z}h_p=0$,
\begin{align*}
 u=H^+h
 \quad\Longleftrightarrow\quad
 Hu=h,\qquad \sum_{p\in\mathcal Z}u_p=0,
\end{align*}
and $H^+\1=0$.

\begin{lemma}\label{lem:mass-derivative}
For distinct sites and positive cell volumes, the cell adjacency
graph is connected. The cell volumes are continuously
differentiable functions of the weights, with
\begin{align}\label{eq:dm/dw}
 \frac{\partial m_p}{\partial w_q}
 =\frac{\partial |C_p(w)|}{\partial w_q}=H_{pq},
 \qquad p,q\in\mathcal Z.
\end{align}
Write $w(m)$ for the optimal weights normalized by
$\int_{\T_L^d}\phi_{w(m)}(y)\dd y=0$.
For directions $h,k\in\R^{\mathcal Z}$ with
$\sum_{p\in\mathcal Z}h_p=\sum_{p\in\mathcal Z}k_p=0$,
define the directional derivatives by
\begin{align*}
  Dw(m)[h]
 &:=\left.\frac{d}{dt}w(m+th)\right|_{t=0}\\
 &=H^+h-\frac1n
 \left(\sum_{p\in\mathcal Z}m_p(H^+h)_p\right)\1,\\
 DF_{\mathcal Z}(m)[h]
 &:=\left.\frac{d}{dt}F_{\mathcal Z}(m+th)\right|_{t=0}
 =\sum_{p\in\mathcal Z}w_p(m)h_p,\\
 D^2F_{\mathcal Z}(m)[h,k]
 &:=\left.\frac{\partial^2}{\partial s\,\partial t}
 F_{\mathcal Z}(m+th+sk)\right|_{s=t=0}
 =h^{\mathsf T}H^+k.
\end{align*}
Here $H$ is evaluated at $w(m)$, and the parameters are taken
small enough that all masses remain positive.
\end{lemma}

\begin{proof}
Keep the sites fixed and perturb the weights to $w+tu$.
The face equation \eqref{eq:face-equation} becomes
\begin{align*}
 2y\cdot(q-p)
 =|q|^2-|p|^2+w_p-w_q+t(u_p-u_q).
\end{align*}
Thus the outward normal speed from the cell at $p$ is
$(u_p-u_q)/(2\ell_F)$. Multiplying by face area and summing gives
\begin{align*}
 \left.\frac{d}{dt}|C_p(w+tu)|\right|_{t=0}
 =\sum_{q\in\mathcal Z\setminus\{p\}}c_{pq}(u_p-u_q)
 =(Hu)_p.
\end{align*}
We next show that $H(w)$ is continuous. Since the sites are fixed,
varying $w$ only translates the hyperplanes bounding the lifted
cells linearly. Therefore, the areas
of faces vary continuously. The distances $\ell_F$ are fixed,
so each $c_{pq}$, and hence $H(w)$, is continuous.

By \eqref{eq:dm/dw} and continuity of $H$, the map $w \to m(w)$ is $C^1$. On the subspace $\sum_pw_p=0$, its
derivative is invertible, so the inverse function
theorem gives $w(m)$ with derivative $H^+$.

To obtain our normalization \eqref{eq:normalization}, subtract
$n^{-1}\int_{\T_L^d}\phi_w(y)\dd y$ from every weight.
Therefore,
\begin{align*}
 Dw(m)[h]
 =H^+h-\frac1n
 \left(\sum_{p\in\mathcal Z}m_p(H^+h)_p\right)\1.
\end{align*}
Finally, since
\begin{align*}
F_{\mathcal Z}(m)
=\sum_{p\in\mathcal Z}m_pw_p(m)
-\int_{\T_L^d}\phi_{w(m)}(y)\dd y,
\end{align*}
the chain rule gives
\begin{align*}
 DF_{\mathcal Z}(m)[h]
 &=\left.\frac{d}{dt}\right|_{t=0}
 \left[
 \sum_{p\in\mathcal Z}(m_p+th_p)w_p(m+th)
 -\int_{\T_L^d}\phi_{w(m+th)}(y)\dd y
 \right]\\
 &=\sum_{p\in\mathcal Z}h_pw_p
   +\sum_{p\in\mathcal Z}m_pDw_p(m)[h]
   -\sum_{p\in\mathcal Z}\int_{C_p(w)}Dw_p(m)[h]\dd y\\
 &=\sum_{p\in\mathcal Z}h_pw_p
   +\sum_{p\in\mathcal Z}\bigl(m_p-|C_p(w)|\bigr)Dw_p(m)[h]\\
 &=\sum_{p\in\mathcal Z}h_pw_p.
\end{align*}
In the second line, the maximizing site is unique almost
everywhere, so the derivative of the potential on $C_p(w)$
is $Dw_p(m)[h]$. The last equality uses $|C_p(w)|=m_p$.

Differentiating once more in direction $k$ gives
\begin{align*}
 D^2F_{\mathcal Z}(m)[h,k]
 &=\sum_{p\in\mathcal Z}h_pDw_p(m)[k]\\
 &=h^{\mathsf T}H^+k,
\end{align*}
since $\sum_ph_p=0$ cancels the constant introduced
by the normalization.
\end{proof}

\subsection{The derivative with respect to a site}
\label{sec:site-derivative}

Keep the masses fixed and move a site $p\in\mathcal Z$. Write
\begin{align*}
 F_{\mathcal Z}^{p,v}(m)
 :=W_2^2\left(
 \sum_{r\in\mathcal Z\setminus\{p\}}m_r\delta_r
 +m_p\delta_{p+v},dy\right).
 \qquad v\in\R^d,
\end{align*}
We write
$\nabla_pF_{\mathcal Z}(m)$ for the gradient obtained by
varying the location of $p$ in this way. Define the corresponding cell barycenter by
\begin{align*}
 \overline y_p
 =\frac1{m_p}\int_{C_p(w)} y \dd y.
\end{align*}

\begin{lemma}\label{lem:site-derivative}
For fixed masses, the optimal cost is Lipschitz in
the site locations. At almost every configuration, for every
$p\in\mathcal Z$,
\begin{align*}
 \nabla_pF_{\mathcal Z}(m)
 &=2\int_{C_p(w)}(p-y)\dd y\\
 &=2m_p(p-\overline y_p).
\end{align*}
Moreover, for every $v\in\R^d$,
\begin{align*}
 F_{\mathcal Z}^{p,v}(m)
 \le F_{\mathcal Z}(m)
 +\nabla_pF_{\mathcal Z}(m)\cdot v+m_p|v|^2.
\end{align*}
\end{lemma}

\begin{proof}
The Lipschitz property and derivative formula follow as in
\cite[Lemma~2.2]{FM26}. The argument applies in every dimension,
after rescaling the torus and replacing the equal masses
by the prescribed masses.

For the inequality, use the original partition as a
competitor. Since
$\rho_L(y,p+v)\le |p+v-y|$, we obtain
\begin{align*}
 F_{\mathcal Z}^{p,v}(m)-F_{\mathcal Z}(m)
 &\le \int_{C_p(w)}
 \bigl(|p+v-y|^2
       -|p-y|^2\bigr)\dd y\\
 &=2v\cdot\int_{C_p(w)}(p-y)\dd y
   +m_p|v|^2\\
 &=\nabla_pF_{\mathcal Z}(m)\cdot v+m_p|v|^2.
\end{align*}
\end{proof}

\section{Random walk on cells}\label{sec:geometry}
\subsection{Replacing a point by transferring mass}
\label{sec:split-paths}

We represent a point replacement by transferring mass
between its old and new locations. The locations remain fixed,
so the mass derivatives from Section~\ref{sec:cells} apply.

For one replacement, fix distinct sites
$z_1,\ldots,z_{n-1},p,q\in\T_L^d$ and define
\begin{align*}
 \nu_s
 &=\sum_{i=1}^{n-1}\delta_{z_i}
   +(1-s)\delta_p+s\delta_q,\\
 F(s)&=W_2^2(\nu_s,\dd y),
 \qquad 0\le s\le1.
\end{align*}
Thus $F(0)$ is the cost with site $p$, and $F(1)$ is the
cost after replacing $p$ by $q$. We call $z_1,\ldots,z_{n-1}$ background sites and $p,q$ transfer sites.

For $0<s<1$, denote the optimal weights by $w(s)$,
with the normalization \eqref{eq:normalization}.
At $s=0$ and $s=1$, the zero-mass site is omitted.

\begin{lemma}\label{lem:split-paths}
For $0<s<1$,
\begin{align*}
F'(s)=w_q(s)-w_p(s).
\end{align*}
Moreover,
\begin{align*}
F(1)-F(0)
=\lim_{\varepsilon\downarrow0}
\int_\varepsilon^{1-\varepsilon}
\bigl(w_q(s)-w_p(s)\bigr)\dd s.
\end{align*}
\end{lemma}

\begin{proof}
For $0<s<1$, all prescribed masses are positive.
Lemma~\ref{lem:mass-derivative}, applied in the
direction $e_q-e_p$, gives
\begin{align*}
F'(s)=w_q(s)-w_p(s).
\end{align*}
The fundamental theorem of calculus therefore
yields, for $0<\varepsilon<1/2$,
\begin{align*}
F(1-\varepsilon)-F(\varepsilon)
=\int_\varepsilon^{1-\varepsilon}
\bigl(w_q(s)-w_p(s)\bigr)\dd s.
\end{align*}

Coupling the common mass identically and moving
the remaining mass between $p$ and $q$ gives
\begin{align*}
W_2^2(\nu_s,\nu_t)
\le |s-t|\rho_L(p,q)^2.
\end{align*}
Hence the triangle inequality yields
\begin{align*}
\left|\sqrt{F(s)}-\sqrt{F(t)}\right|
&\le W_2(\nu_s,\nu_t)\\
&\le |s-t|^{1/2}\rho_L(p,q).
\end{align*}
Thus $F$ is continuous on $[0,1]$.
Letting $\varepsilon\downarrow0$ proves
the integral identity.
\end{proof}

\begin{remark}
For every $s\in(0,1)$, the optimal cells have
positive masses and their adjacency graph
is connected. This guarantees the existence
and uniqueness of $H(w(s))^+$ as the inverse
on the zero-sum subspace, extended by zero
on constants.
The inverse may become unbounded near the
endpoints. We recover the endpoint costs
through the limits in Lemma~\ref{lem:split-paths},
without inverting $H$ there.
\end{remark}

When $z_1,\ldots,z_{n-1},p,q$ are independent
and uniform on $\T_L^d$, Efron--Stein gives
\begin{align*}
\Var(C_L)\le\frac n2\E|F(1)-F(0)|^2.
\end{align*}
We therefore want to bound
$\E|F(1)-F(0)|^2$ by some constant $C_d$.
The following lemma connects this quantity
to the derivatives of the optimal weights.

\begin{lemma}\label{lem:replacement-response-bound}
For the one-replacement interpolation,
\begin{align}\label{eq:replacement-response-bound}
\left(\E|F(1)-F(0)|^2\right)^{1/2}
&\le
\left(\E|w_q(1)|^2\right)^{1/2}
+\left(\E|w_p(0)|^2\right)^{1/2}\notag\\
&\quad+\int_0^1
\left(
\E\left[
\sum_{r\in\mathcal Z}m_r(s)
\left|\frac{\dd w_r(s)}{\dd s}\right|^2
\right]
\right)^{1/2}\dd s.
\end{align}
\end{lemma}

\begin{proof}
We may assume that the right-hand side is finite.
The fundamental theorem of calculus gives
\begin{align*}
w_q(s)&=w_q(1)-\int_s^1w_q'(u)\dd u,\\
w_p(s)&=w_p(0)+\int_0^s w_p'(u)\dd u.
\end{align*}
Substituting into Lemma~\ref{lem:split-paths}
and changing the order of integration yields
\begin{align*}
F(1)-F(0)
&=w_q(1)-w_p(0)\\
&\quad-\int_0^1
\bigl[s\,w_q'(s)+(1-s)w_p'(s)\bigr]\dd s.
\end{align*}
Since $m_q(s)=s$ and $m_p(s)=1-s$,
Jensen's inequality gives
\begin{align*}
\bigl|s\,w_q'(s)+(1-s)w_p'(s)\bigr|^2
&\le s|w_q'(s)|^2+(1-s)|w_p'(s)|^2\\
&\le\sum_{r\in\mathcal Z}m_r(s)|w_r'(s)|^2.
\end{align*}
Minkowski's inequality then gives
\begin{align*}
\left(\E|F(1)-F(0)|^2\right)^{1/2}
&\le
\left(\E|w_q(1)|^2\right)^{1/2}
+\left(\E|w_p(0)|^2\right)^{1/2}\\
&\quad+\int_0^1
\left(
\E|s\,w_q'(s)+(1-s)w_p'(s)|^2
\right)^{1/2}\dd s\\
&\le
\left(\E|w_q(1)|^2\right)^{1/2}
+\left(\E|w_p(0)|^2\right)^{1/2}\\
&\quad+\int_0^1
\left(
\E\sum_{r\in\mathcal Z}m_r(s)|w_r'(s)|^2
\right)^{1/2}\dd s,
\end{align*}
which proves \eqref{eq:replacement-response-bound}.
\end{proof}

The difficult part is to bound the integral in
\eqref{eq:replacement-response-bound}.
Once this is established, the endpoint moments
$\E|w_q(1)|^2$ and $\E|w_p(0)|^2$ follow from
the cell-radius estimates in Appendix \ref{app:radius-moments} and elementary
arguments in the proof of
Corollary~\ref{cor:variance-upper-bound}.

For two replacements, fix distinct sites
$z_1,\ldots,z_{n-2},p_1,q_1,p_2,q_2\in\T_L^d$
and define
\begin{align}\label{eq:two-transfer-measure}
 \nu_{s,t}
 &=\sum_{i=1}^{n-2}\delta_{z_i}
   +(1-s)\delta_{p_1}+s\delta_{q_1}
   +(1-t)\delta_{p_2}+t\delta_{q_2},\notag\\
 F(s,t)&=W_2^2(\nu_{s,t},\dd y),
 \qquad 0\le s,t\le1.
\end{align}
For $0<s,t<1$, denote the optimal weights by $w(s,t)$
and write $H_{s,t}:=H(w(s,t))$. On the boundary,
the zero-mass sites are omitted.

\begin{lemma}\label{lem:two-split-paths}
For $0<s,t<1$,
\begin{align}\label{eq:two-transfer-mixed-derivative}
 \partial_s\partial_t F(s,t)
 &=(e_{q_1}-e_{p_1})^{\mathsf T}
 H_{s,t}^{+}(e_{q_2}-e_{p_2}).
\end{align}
Moreover,
\begin{align}\label{eq:mixed-replacement-path}
 &F(1,1)-F(1,0)-F(0,1)+F(0,0)\notag\\
 &\qquad=
 \lim_{\varepsilon\downarrow0}
 \int_\varepsilon^{1-\varepsilon}
 \int_\varepsilon^{1-\varepsilon}
 (e_{q_1}-e_{p_1})^{\mathsf T}
 H_{s,t}^{+}(e_{q_2}-e_{p_2})
 \dd t\,\dd s.
\end{align}
\end{lemma}

\begin{proof}
For $0<s,t<1$, all masses are positive.
Lemma~\ref{lem:mass-derivative}, applied in the
directions $e_{q_1}-e_{p_1}$ and $e_{q_2}-e_{p_2}$,
gives \eqref{eq:two-transfer-mixed-derivative}.
The fundamental theorem of calculus yields
\begin{align*}
 &\int_\varepsilon^{1-\varepsilon}
 \int_\varepsilon^{1-\varepsilon}
 \partial_s\partial_t F(s,t)\dd t\,\dd s\\
 &\qquad=
 F(1-\varepsilon,1-\varepsilon)
 -F(1-\varepsilon,\varepsilon)
 -F(\varepsilon,1-\varepsilon)
 +F(\varepsilon,\varepsilon).
\end{align*}
Continuity of the cost up to the mass boundary
then gives \eqref{eq:mixed-replacement-path}.
\end{proof}

\subsection{The random walk on cells}
\label{sec:cell-walk}
Fix sites and masses as in Section~\ref{sec:split-paths}.
For one replacement, the site set is
\begin{align*}
 \mathcal Z=\{z_1,\ldots,z_{n-1},p,q\},
\end{align*}
with masses $m_{z_i}=1$, $m_p=1-s$, and $m_q=s$.
For two replacements, the site set is
\begin{align*}
 \mathcal Z=\{z_1,\ldots,z_{n-2},p_1,q_1,p_2,q_2\},
\end{align*}
with masses
\begin{align*}
 m_{z_i}=1,\qquad
 m_{p_1}=1-s,\quad m_{q_1}=s,\qquad
 m_{p_2}=1-t,\quad m_{q_2}=t.
\end{align*}

Consider the continuous-time random walk $(X_t)_{t\ge0}$
on $\mathcal Z$ that jumps from $r$ to $u$ at rate
$c_{ru}/m_r$. Its generator acts on
$f\in\R^{\mathcal Z}$ by
\begin{align}\label{eq:generator}
 (\mathcal Lf)_r
 :=\frac1{m_r}\sum_{u\in\mathcal Z}c_{ru}(f_u-f_r)
 =-\frac1{m_r}(Hf)_r.
\end{align}
Since $c_{ru}=c_{ur}$,
\begin{align*}
 m_r\frac{c_{ru}}{m_r}
 =m_u\frac{c_{ur}}{m_u}.
\end{align*}
Thus the walk is reversible with respect to the cell
volumes, and its stationary distribution is
$(m_r/n)_{r\in\mathcal Z}$.
Let $\mathbb P_r$ denote the law of the walk started
at $r$, conditional on the sites and masses. Define
\begin{align*}
 p_t(r,u)&:=\frac{\mathbb P_r(X_t=u)}{m_u},\\
 \overline p_t(r,u)&:=p_t(r,u)-\frac1n.
\end{align*}
Reversibility and conservation of probability give
\begin{align*}
 p_t(r,u)=p_t(u,r),
 \qquad
 \sum_{u\in\mathcal Z}m_up_t(r,u)=1.
\end{align*}
Since the cell adjacency graph is connected,
\begin{align*}
 p_t(r,u)\longrightarrow\frac1n
 \qquad\text{as }t\to\infty.
\end{align*}

For one replacement, the walk represents the weight
and potential derivatives as follows.

\begin{lemma}\label{lem:walk-response}
For $0<s<1$ and $r\in\mathcal Z$,
\begin{align*}
 \frac{d}{ds}w_r(s)
 =\int_0^\infty
 \bigl(p_t(r,q)-p_t(r,p)\bigr)\dd t.
\end{align*}
Moreover, let $\phi_s = \phi_{w(s)}$ and the potential derivative
$\psi_s:=\frac{d}{ds}\phi_s$ satisfies, almost everywhere,
\begin{align*}
 \psi_s(y)&=\frac{d}{ds}w_r(s),
 \qquad y\in C_r(w(s)),\\
 \|\psi_s\|_{L^2(\T_L^d)}^2
 &=\sum_{r\in\mathcal Z}m_r
 \left|\frac{d}{ds}w_r(s)\right|^2.
\end{align*}
\end{lemma}

\begin{proof}
Differentiating the weight in Lemma~\ref{lem:mass-derivative} and the normalization gives
\begin{align*}
 H\frac{d}{ds}w(s)=e_q-e_p,
 \qquad
 \sum_{r\in\mathcal Z}m_r\frac{d}{ds}w_r(s)=0.
\end{align*}
Hence
\begin{align*}
 \left(\mathcal L\frac{d}{ds}w(s)\right)_u
 =-\frac{\1_{\{u=q\}}}{m_q}
  +\frac{\1_{\{u=p\}}}{m_p}.
\end{align*}
Dynkin's formula yields
\begin{align*}
 \frac{d}{ds}w_r(s)
 -\mathbb E_r\left[
 \frac{d}{ds}w_{X_T}(s)
 \right]
 &=\mathbb E_r\left[
 \int_0^T
 \left(
 \frac{\1_{\{X_t=q\}}}{m_q}
 -\frac{\1_{\{X_t=p\}}}{m_p}
 \right)\dd t
 \right]\\
 &=\int_0^T
 \bigl(p_t(r,q)-p_t(r,p)\bigr)\dd t.
\end{align*}
The walk is finite and irreducible, so its distribution
converges to the stationary distribution $(m_u/n)_{u\in\mathcal Z}$.
Therefore,
\begin{align*}
\mathbb E_r\left[
 \frac{d}{ds}w_{X_T}(s)
\right]
\xrightarrow[T\to\infty]{}
\frac1n\sum_{u\in\mathcal Z}m_u\frac{d}{ds}w_u(s)=0.
\end{align*}
Letting $T\to\infty$ proves the first identity.

At almost every $y\in C_r(w(s))$, the maximizing site
in the definition of $\phi_s(y)$ is unique.
Differentiating gives
$\psi_s(y)=\frac{d}{ds}w_r(s)$.
Since $|C_r|=m_r$,
\begin{align*}
 \|\psi_s\|_{L^2(\T_L^d)}^2
 &=\sum_{r\in\mathcal Z}\int_{C_r}
 \left|\frac{d}{ds}w_r(s)\right|^2\dd y\\
 &=\sum_{r\in\mathcal Z}m_r
 \left|\frac{d}{ds}w_r(s)\right|^2.
\end{align*}
\end{proof}

By symmetry and the semigroup property, for every
$o\in\mathcal Z$,
\begin{align*}
 \sum_{r\in\mathcal Z}m_r
 \left(\int_0^\infty\overline p_t(o,r)\dd t\right)^2
 &=\int_0^\infty\int_0^\infty
 \overline p_{a+b}(o,o)\dd a\,\dd b\\
 &=\int_0^\infty t\,\overline p_t(o,o)\dd t.
\end{align*}
For one replacement, Lemma~\ref{lem:walk-response} and
$(a-b)^2\le2a^2+2b^2$ therefore give
\begin{align}\label{eq:response-return-bound}
 \|\psi_s\|_{L^2(\T_L^d)}^2
 &=\sum_{r\in\mathcal Z}m_r
 \left|
 \int_0^\infty
 \bigl(\overline p_t(q,r)-\overline p_t(p,r)\bigr)\dd t
 \right|^2\notag\\
 &\le2\sum_{o\in\{p,q\}}\sum_{r\in\mathcal Z}m_r
 \left(\int_0^\infty\overline p_t(o,r)\dd t\right)^2
 \notag\\
 &=2\int_0^\infty t
 \bigl(\overline p_t(q,q)+\overline p_t(p,p)\bigr)\dd t.
\end{align}
Our next goal is to bound the centered return density
$\overline p_t(o,o)$ for both the one-replacement and
two-replacement configurations.

\subsection{Poincar\'e inequality and large-time decay}
\label{sec:large-time}

We first control the return density for large times.
The following inequality is a discrete analogue of
the Poincar\'e inequality on $\T_L^d$ for functions
that are constant on each transport cell.
We prove it by comparing their values along line
segments, a method that we will also use later.

\begin{lemma}[Poincar\'e inequality]
\label{lem:cell-poincare}
For every $f\in\R^{\mathcal Z}$,
\begin{align}\label{eq:cell-poincare}
 \sum_{r\in\mathcal Z}m_r|f_r-\overline f_m|^2
 \le C_dL^2f^{\mathsf T}Hf,
 \qquad
 \overline f_m:=\frac1n\sum_{r\in\mathcal Z}m_rf_r.
\end{align}
\end{lemma}
\begin{proof}
Define $V(y)=f_r$ on $C_r$, so that
\begin{align*}
\int_{\T_L^d}V(y)\dd y
=\sum_{r\in\mathcal Z}m_rf_r
=n\overline f_m.
\end{align*}

Fix $h\in[-L/2,L/2]^d\setminus\{0\}$ and put
$e=h/|h|$. Follow the segment from a lift of $y$
to $y+h$ through the lifted cells.
At a crossing from a cell with generator $a$
into one with generator $b$, recall
\begin{align*}
\ell_F=|b-a|,
\qquad
\nu_F=\frac{b-a}{|b-a|}.
\end{align*}
The vector $\nu_F$ is normal to their common face
and points from the former cell into the latter,
so $\nu_F\cdot e>0$.

Let $a_0$ and $a_N$ be the generators of the cells
containing the two endpoints. The generator increments
telescope, giving
\begin{align*}
\sum_{\text{crossed }F}\ell_F(\nu_F\cdot e)
=(a_N-a_0)\cdot e \leq C_dL.
\end{align*}

Writing $\Delta_Ff$ for the change in cell value
at a crossing, Cauchy--Schwarz gives, for almost
every $y$,
\begin{align*}
|V(y+h)-V(y)|^2
&=\left|\sum_{\text{crossed }F}\Delta_Ff\right|^2\\
&\le
\left(\sum_{\text{crossed }F}\ell_F(\nu_F\cdot e)\right)
\left(\sum_{\text{crossed }F}
\frac{|\Delta_Ff|^2}{\ell_F(\nu_F\cdot e)}\right)\\
&\le C_dL\sum_{\text{crossed }F}
\frac{|\Delta_Ff|^2}{\ell_F(\nu_F\cdot e)}.
\end{align*}
If the segment stays in one cell, both sides vanish.

Integrating over $y$ counts crossings of each face
piece $F$ with total multiplicity
$|h|\,|\nu_F\cdot e|\,\mathcal H^{d-1}(F)$.
Consequently,
\begin{align*}
\int_{\T_L^d}|V(y+h)-V(y)|^2\dd y
&\le C_dL|h|
\sum_F\frac{\mathcal H^{d-1}(F)}{\ell_F}
|\Delta_Ff|^2\\
&\le C_dL|h|\,f^{\mathsf T}Hf.
\end{align*}
The sum counts each face piece once. Averaging over $h\in[-L/2,L/2]^d$ yields
\begin{align*}
2\int_{\T_L^d}|V(y)-\overline f_m|^2\dd y
&=\frac1n\int_{[-L/2,L/2]^d}
\int_{\T_L^d}|V(y+h)-V(y)|^2\dd y\,\dd h\\
&\le C_dL^2f^{\mathsf T}Hf.
\end{align*}
Since $|C_r|=m_r$, this proves
\eqref{eq:cell-poincare}.
\end{proof}

Fix $o\in\mathcal Z$ and write
$f_t(r)=\overline p_t(o,r)$. Symmetry and the
semigroup property give
\begin{align*}
\overline p_{2t}(o,o)
=\sum_{r\in\mathcal Z}m_rf_t(r)^2,
\qquad
\overline f_t = \sum_{r\in\mathcal Z}m_rf_t(r)=0.
\end{align*}
Since $\partial_tf_t=\mathcal Lf_t$ and by \eqref{eq:generator}, we obtain
\begin{align}\label{eq:heat-energy}
\frac{d}{dt}\overline p_{2t}(o,o)
&=2\sum_{r\in\mathcal Z}m_rf_t(r)
  (\mathcal Lf_t)_r\notag\\
&=-2f_t^{\mathsf T}Hf_t.
\end{align}
The Poincar\'e inequality \eqref{eq:cell-poincare} therefore implies
\begin{align*}
\frac{d}{dt}\overline p_{2t}(o,o)
\le-\frac{c_d}{L^2}\overline p_{2t}(o,o).
\end{align*}
Integrating and adjusting $c_d$ gives
\begin{align}\label{eq:return-large-time}
\overline p_t(o,o)
\le e^{-c_d(t-t_0)/L^2}\overline p_{t_0}(o,o),
\qquad t\ge t_0\ge0.
\end{align}
Thus it remains to estimate the return density
up to time $L^2$.

\subsection{Small-time estimates}
\label{sec:small-time}

We next estimate the return density for $t\le L^2$.
We first control the probability that the walk
leaves a ball.

For each $r\in\mathcal Z$, define
\begin{align*}
R_r:=\sup_{y\in C_r}\rho_L(y,r).
\end{align*}
Fix $o\in\mathcal Z$ and
$1\le\ell\le R/16$, with $R\le L/16$.
Let $G(o;R,\ell)$ be the event that
\begin{align*}
R_r\le\ell
\quad\text{whenever}\quad
r\in B_{2R}(o)
\ \text{or}\ 
C_r\cap B_{2R}(o)\ne\varnothing.
\end{align*}
Let
\begin{align*}
\tau_R:=\inf\{t\ge0:\rho_L(X_t,o)\ge R\}
\end{align*}
be the first exit time from $B_R(o)$.
On $G(o;R,\ell)$, every cell visited before
$\tau_R$ has radius at most $\ell$ and lies in
$B_{R+\ell}(o)$. Consequently, every jump up to
and including the exit has length at most $2\ell$.

Our plan is to separate paths that remain in $B_R(o)$
from those that leave it. Since
$m_o\overline p_t(o,o)=\mathbb P_o(X_t=o)-m_o/n$,
\begin{align}\label{eq:return-exit-decomposition}
m_o\overline p_t(o,o)
&\le \mathbb P_o(X_t=o,\ t<\tau_R)
   +\mathbb P_o(X_t=o,\ \tau_R\le t)\notag\\
&\le \mathbb P_o(X_t=o,\ t<\tau_R)
   +\mathbb P_o(\tau_R\le t).
\end{align}
We will estimate the first term using the walk killed
upon leaving the ball. For the second term, we first
prove an exit estimate and then apply it successively
to smaller balls to obtain a sufficiently high
polynomial decay rate. Both terms must decay fast
enough to control the integral
in \eqref{eq:response-return-bound}.

\begin{lemma}\label{lem:walk-exit}
On $G(o;R,\ell)$, for every $t\ge0$,
\begin{align}\label{eq:walk-exit}
\mathbb P_o(\tau_R\le t)
\le\frac{2dt+6\ell^2}{R^2}.
\end{align}
\end{lemma}

\begin{proof}
Work on $G(o;R,\ell)$ and use consistent lifts of
the sites and cells near $o$ to $\R^d$.
We use the same symbols for these lifts.

Consider a cell $C_r$ with $r\in B_R(o)$, and let
$\nu(y)$ denote its outward unit normal at
$y\in\partial C_r$. On a face $F$ shared with $C_u$,
\begin{align*}
\nu(y)=\nu_F=\frac{u-r}{|u-r|}.
\end{align*}
By the definition of $c_{ru}$,
\begin{align*}
\sum_u c_{ru}(u-r)
&=\frac12\sum_{F\subset\partial C_r}
\mathcal H^{d-1}(F)\nu_F\\
&=\frac12\int_{\partial C_r}\nu(y)\,
\dd\mathcal H^{d-1}(y)
=0.
\end{align*}
The last equality follows by applying the divergence
theorem. For $y$ on the face $F$ shared by $C_r$ and $C_u$,
both sites attain the maximum defining $\phi_w(y)$.
Hence
\begin{align*}
w_r-|y-r|^2=w_u-|y-u|^2.
\end{align*}
Expanding and rearranging gives
\begin{align*}
2(y-r)\cdot(u-r)=|u-r|^2+w_r-w_u.
\end{align*}
Integrating over the faces and applying the
divergence theorem again to $y-r$ yields
\begin{align*}
d m_r
&=\int_{\partial C_r}(y-r)\cdot\nu\,
  \dd\mathcal H^{d-1}\\
&=\sum_u c_{ru}\bigl(|u-r|^2+w_r-w_u\bigr)\\
&=\sum_u c_{ru}|u-r|^2+(Hw)_r.
\end{align*}
For the fixed point $o$, put $f(u)=|u-o|^2-w_u$.
Since
\begin{align*}
|u-o|^2-|r-o|^2
=|u-r|^2+2(r-o)\cdot(u-r),
\end{align*}
the preceding identities give
\begin{align}\label{eq:corrected-distance}
(\mathcal Lf)_r
&=\frac1{m_r}\sum_u c_{ru}
\bigl(|u-o|^2-|r-o|^2-w_u+w_r\bigr)
\notag\\
&=\frac1{m_r}\left(
\sum_u c_{ru}|u-r|^2
+2(r-o)\cdot\sum_u c_{ru}(u-r)
+(Hw)_r
\right)
\notag\\
&=\frac1{m_r}\left(
\sum_u c_{ru}|u-r|^2+(Hw)_r
\right)
=d.
\end{align}

We next control the weight difference.
For almost every $y\in B_{2R}(o)$, let $C_u$
be the cell containing $y$. Since this cell meets
$B_{2R}(o)$, we have
\begin{align*}
|\nabla\phi_w(y)|
=2\rho_L(y,u)
\le2R_u
\le2\ell.
\end{align*}
Thus $\phi_w$ is $2\ell$-Lipschitz there.
For any site $r$ visited up to and including
the exit, choose $y\in C_r$. Since
$|y-r|\le\ell$ and
$\phi_w(y)=w_r-|y-r|^2$, we have
\begin{align*}
0\le\phi_w(r)-w_r
&=\phi_w(r)-\phi_w(y)-|y-r|^2\\
&\le2\ell|y-r|-|y-r|^2
\le\ell^2.
\end{align*}
The same holds at $o$. Writing
$w_r=\phi_w(r)-(\phi_w(r)-w_r)$ and applying
the triangle inequality gives
\begin{align*}
|w_r-w_o|\le2\ell|r-o|+\ell^2.
\end{align*}

Put $T=t\wedge\tau_R$. Dynkin's formula applied
to \eqref{eq:corrected-distance} and
$2xy\le\frac12x^2+2y^2$ gives
\begin{align*}
\mathbb E_o|X_T-o|^2
&=d\mathbb E_oT+\mathbb E_o(w_{X_T}-w_o)\\
&\le dt+2\ell\mathbb E_o|X_T-o|+\ell^2\\
&\le dt+\frac12\mathbb E_o|X_T-o|^2+3\ell^2.
\end{align*}
Hence
\begin{align*}
\mathbb E_o|X_T-o|^2\le2dt+6\ell^2.
\end{align*}
On $\{\tau_R\le t\}$, we have $|X_T-o|\ge R$.
It follows that
\begin{align*}
R^2\mathbb P_o(\tau_R\le t)
\le\mathbb E_o|X_T-o|^2
\le2dt+6\ell^2,
\end{align*}
as claimed.
\end{proof}

\begin{remark}
The function $f(u)=|u-o|^2-w_u$ plays the role of
a corrected squared distance for the random walk.
Indeed, \eqref{eq:corrected-distance} gives
$\mathcal Lf=d$, just as
\begin{align*}
\frac12\Delta|x-o|^2=d
\end{align*}
for Brownian motion in $\R^d$.
The weight correction compensates for the geometry
of the transport cells.
\end{remark}

We strengthen the exit estimate by applying it
successively to smaller balls.

\begin{corollary}\label{cor:walk-exit-iterated}
Let $k\ge1$ be an integer and suppose
$\ell\le R/(16k)$. On $G(o;R,\ell)$,
\begin{align}\label{eq:walk-exit-iterated}
\mathbb P_o(\tau_R\le t)
\le
\left(\frac{C_dk^2(t+\ell^2)}{R^2}\right)^k,
\qquad t\ge0.
\end{align}
\end{corollary}

\begin{proof}
Put $a=R/(2k)$ and define successive exit times by
\begin{align*}
\sigma_0&=0,\\
\sigma_{j+1}
&=\inf\{v\ge\sigma_j:
\rho_L(X_v,X_{\sigma_j})\ge a\}.
\end{align*}
Each jump has length at most $2\ell$, so the
displacement during each such step is at most
$a+2\ell$. Up to the $k$th exit, the walk therefore
stays within distance
\begin{align*}
k(a+2\ell)\le\frac{5R}{8}
\end{align*}
of $o$. Thus leaving $B_R(o)$ requires completing
all $k$ smaller exits, and
\begin{align*}
\{\tau_R\le t\}\subseteq\{\sigma_k\le t\}.
\end{align*}

All these smaller balls and their adjacent cells
remain in the region controlled by $G(o;R,\ell)$.
The proof of Lemma~\ref{lem:walk-exit}, centered at
$X_{\sigma_j}$, and the strong Markov property give
\begin{align*}
\mathbb P_o\left(
\sigma_{j+1}-\sigma_j\le t
\,\middle|\,\mathcal F_{\sigma_j}
\right)
\le\frac{C_dk^2(t+\ell^2)}{R^2}.
\end{align*}
If $\sigma_k\le t$, each of these $k$ durations
is at most $t$. Successive conditioning therefore
yields
\begin{align*}
\mathbb P_o(\tau_R\le t)
&\le\mathbb P_o(\sigma_k\le t)\\
&\le
\mathbb P_o\left(
\bigcap_{j=0}^{k-1}
\{\sigma_{j+1}-\sigma_j\le t\}
\right)\\
&\le
\left(\frac{C_dk^2(t+\ell^2)}{R^2}\right)^k.
\end{align*}
\end{proof}

We next prove the smoothing estimate used to control
the first term in \eqref{eq:return-exit-decomposition} using similar method as in Lemma~\ref{lem:cell-poincare}.

\begin{lemma}\label{lem:local-smoothing}
On $G(o;R,\ell)$, let $f\in\R^{\mathcal Z}$ satisfy
$f_r=0$ for $r\notin B_R(o)$. For every $0<a\le L/8$,
\begin{align}\label{eq:local-smoothing}
\sum_{r\in\mathcal Z}m_rf_r^2
&\le C_da^{-d}
\left(\sum_{r\in\mathcal Z}m_r|f_r|\right)^2
\notag\\
&\quad+C_da(a+\ell)f^{\mathsf T}Hf.
\end{align}
\end{lemma}

\begin{proof}
Define $V(y)=f_r$ on $C_r$. We first show that
\begin{align}\label{eq:local-translation}
\|V(\cdot+h)-V\|_{L^2}^2
\le C_d|h|(|h|+\ell)f^{\mathsf T}Hf.
\end{align}

Use the crossing argument from the proof of
Lemma~\ref{lem:cell-poincare}. At each crossing where
the cell value changes, at least one incident site
belongs to $B_R(o)$. Its cell lies in
$B_{R+\ell}(o)$, so the adjacent cell meets
$B_{2R}(o)$. Both cells therefore have radius
at most $\ell$.

By the same argument as in Lemma~\ref{lem:cell-poincare},
\begin{align*}
\sum_{\substack{\text{crossed }F\\\Delta_Ff\ne0}}
\ell_F(\nu_F\cdot e)
\le |h|+2\ell.
\end{align*}
Cauchy--Schwarz therefore gives
\begin{align*}
|V(y+h)-V(y)|^2
\le (|h|+2\ell)
\sum_{\substack{\text{crossed }F\\\Delta_Ff\ne0}}
\frac{|\Delta_Ff|^2}{\ell_F(\nu_F\cdot e)}.
\end{align*}
If there is no nonzero value jump, the left side
vanishes. Integrating over $y$ counts crossings of
each face piece $F$ with total multiplicity
$|h|\,|\nu_F\cdot e|\,\mathcal H^{d-1}(F)$.
Therefore,
\begin{align*}
\|V(\cdot+h)-V\|_{L^2}^2
&\le |h|(|h|+2\ell)
\sum_F\frac{\mathcal H^{d-1}(F)}{\ell_F}
|\Delta_Ff|^2\\
&=2|h|(|h|+2\ell)f^{\mathsf T}Hf\\
&\le C_d|h|(|h|+\ell)f^{\mathsf T}Hf.
\end{align*}
This proves \eqref{eq:local-translation}.

Let $\eta_a$ be a smooth probability density supported
in $B_a(0)$, with
$\|\eta_a\|_{L^2}^2\le C_da^{-d}$.
Jensen's inequality and \eqref{eq:local-translation}
give
\begin{align*}
\|V-\eta_a*V\|_{L^2}^2
&=\int_{\T_L^d}
\left|\int_{\T_L^d}\eta_a(h)
\bigl(V(y)-V(y-h)\bigr)\dd h\right|^2\dd y\\
&\le\int_{\T_L^d}\int_{\T_L^d}\eta_a(h)
|V(y)-V(y-h)|^2\dd h\,\dd y\\
&=\int_{\T_L^d}\eta_a(h)
\|V-V(\cdot-h)\|_{L^2}^2\dd h\\
&\le C_da(a+\ell)f^{\mathsf T}Hf.
\end{align*}
Also, Young's convolution inequality gives
\begin{align*}
\|\eta_a*V\|_{L^2}^2
\le\|\eta_a\|_{L^2}^2\|V\|_{L^1}^2
\le C_da^{-d}\|V\|_{L^1}^2.
\end{align*}
Combining these estimates with
\begin{align*}
\|V\|_{L^2}^2
&\le2\|V-\eta_a*V\|_{L^2}^2
  +2\|\eta_a*V\|_{L^2}^2,\\
\|V\|_{L^2}^2&=\sum_r m_rf_r^2,
\qquad
\|V\|_{L^1}=\sum_r m_r|f_r|,
\end{align*}
proves the lemma.
\end{proof}

We now estimate the return probability of the walk
killed upon leaving $B_R(o)$. Its transition density is
\begin{align*}
p_t^R(r,u)
:=\frac{\mathbb P_r(X_t=u,\ t<\tau_R)}{m_u},
\qquad r,u\in\mathcal Z\cap B_R(o).
\end{align*}
We extend this density by zero outside the ball.

\begin{lemma}\label{lem:killed-return}
On $G(o;R,\ell)$, for every $0<a\le L/8$ and $t\ge0$,
\begin{align}\label{eq:killed-return}
\mathbb P_o(X_t=o,\ t<\tau_R)
\le C_dm_oa^{-d}
+\exp\left(-\frac{c_dt}{a(a+\ell)}\right).
\end{align}
\end{lemma}

\begin{proof}
Put $f_t(r)=p_t^R(o,r)$ and
\begin{align*}
Z(t):=\sum_r m_rf_t(r)^2.
\end{align*}
The killed kernel is symmetric and satisfies the
semigroup identity, so
\begin{align*}
Z(t)=p_{2t}^R(o,o),
\qquad Z(0)=m_o^{-1}.
\end{align*}
Its total mass is the survival probability:
\begin{align*}
\sum_r m_rf_t(r)=\mathbb P_o(t<\tau_R)\le1.
\end{align*}
Also note that
\begin{align*}
Z'(t)
=2\sum_{r\in\mathcal Z}m_rf_t(r)\partial_t f_t(r)
=-2f_t^{\mathsf T}Hf_t.
\end{align*}

Fix $a$. Applying \eqref{eq:local-smoothing} to $f_t$
therefore yields
\begin{align*}
Z(t)
&\le C_da^{-d}
+C_da(a+\ell)f_t^{\mathsf T}Hf_t\\
&=C_da^{-d}-\frac{C_d}{2}a(a+\ell)Z'(t).
\end{align*}
Rearranging,
\begin{align*}
Z'(t)
\le-\frac{c_d}{a(a+\ell)}Z(t)
+\frac{C_da^{-d}}{a(a+\ell)}.
\end{align*}
Solving this linear differential inequality gives
\begin{align*}
Z(t)
\le m_o^{-1}\exp\left(-\frac{c_dt}{a(a+\ell)}\right)
+C_da^{-d}.
\end{align*}
Finally,
\begin{align*}
\mathbb P_o(X_t=o,\ t<\tau_R)
=m_op_t^R(o,o)=m_oZ(t/2).
\end{align*}
Substituting the preceding bound and adjusting $c_d$
proves \eqref{eq:killed-return}.
\end{proof}

Combining \eqref{eq:return-exit-decomposition},
\eqref{eq:walk-exit-iterated}, and
\eqref{eq:killed-return}, we obtain the following
bound on $G(o;R,\ell)$. For every integer $k\ge1$
with $\ell\le R/(16k)$ and every $0<a\le L/8$,
\begin{align}\label{eq:return-good-event}
m_o\overline p_t(o,o)
&\le C_dm_oa^{-d}
+\exp\left(-\frac{c_dt}{a(a+\ell)}\right)
\notag\\
&\quad+
\left(\frac{C_dk^2(t+\ell^2)}{R^2}\right)^k.
\end{align}
The first two terms bound the return probability
before the walk leaves the ball, and the last
bounds the probability of leaving it.

We fix $k$ depending only on $d$ and choose
$a,\ell,R$ as functions of $t$,
and estimate the probability of $G(o;R,\ell)^c$
to obtain a return bound averaged over the random
sites.

Fix either transfer configuration from
Section~\ref{sec:split-paths}. We fix the transfer
locations and may also fix the location of one
background site. The remaining background sites
are independent and uniform on $\T_L^d$.
The root $o$ is either a transfer site or a background
site whose location we also fix.
All probabilities and expectations over configurations
below are conditional on these fixed locations.

\begin{lemma}\label{lem:good-event-probability}
For every $A>0$,
\begin{align}\label{eq:good-event-probability}
\mathbb P\bigl(G(o;R,\ell)^c\bigr)
\le C_{d,A}(1+R)^d\ell^{-A}.
\end{align}
The bound holds for either transfer configuration,
uniformly in the positive transfer masses
and the locations of all fixed transfer and
background sites.
\end{lemma}

\begin{proof}
If $G(o;R,\ell)$ fails, some site $u$ satisfies
$R_u>\ell$ and either lies in $B_{2R}(o)$ or has
a cell meeting that ball. In either case,
\begin{align*}
\rho_L(u,o)\le2R+R_u.
\end{align*}
Choose $j\ge0$ such that
$2^j\ell<R_u\le2^{j+1}\ell$. Then
\begin{align*}
u\in B_{2R+2^{j+1}\ell}(o),
\qquad R_u>2^j\ell.
\end{align*}

Choose $M>\max\{A,d\}$. By
\eqref{eq:conditional-radius-moments} and Markov's
inequality, for every site $u$,
\begin{align*}
\mathbb P(R_u>b\mid u)
\le b^{-M}\mathbb E[R_u^M\mid u]
\le C_{d,M}b^{-M},
\qquad b\ge1.
\end{align*}

For each random background site $u$ and measurable
set $E\subseteq\T_L^d$,
\begin{align*}
\mathbb P(u\in E,\ R_u>b)
&=\frac1n\int_E
\mathbb P(R_u>b\mid u=x)\dd x\\
&\le\frac{C_{d,M}}n|E|b^{-M}.
\end{align*}
The union bound over sites and scales therefore gives
\begin{align*}
\mathbb P(G(o;R,\ell)^c)
&\le C_{d,M}\ell^{-M}\\
&\quad+\sum_{u\text{ random}}\sum_{j=0}^\infty
\mathbb P\left(
u\in B_{2R+2^{j+1}\ell}(o),\
R_u>2^j\ell
\right)\\
&\le C_{d,M}\ell^{-M}
+C_{d,M}\sum_{j=0}^\infty
(R+2^j\ell)^d(2^j\ell)^{-M}\\
&\le C_{d,M}(1+R)^d\ell^{-M}
\sum_{j=0}^\infty2^{-j(M-d)}\\
&\le C_{d,A}(1+R)^d\ell^{-A}.
\end{align*}
The first term accounts for the fixed sites: at most
four transfer sites and one background site, which
may be the root $o$. By
\eqref{eq:conditional-radius-moments} and Markov's
inequality, each contributes at most
$C_{d,M}\ell^{-M}$.

For a random background site $u$, conditioning on
its location makes it an additional prescribed unit
atom in Lemma~\ref{lem:conditional-radius-moments}.
Thus that lemma is used with at most two transfers
and two prescribed unit atoms, so its dependence on
$k+j$ is absorbed into $C_{d,M}$.
Integrating over $u$ gives a factor $1/n$, which
is canceled by summing over at most $n$ random
background sites. This proves the second inequality.

The last two inequalities use $\ell\le R$,
$\ell\ge1$, and $M>\max\{A,d\}$.
\end{proof}

We now combine the preceding estimates. Put
\begin{align*}
\alpha:=\frac{7d+4}{16},
\qquad
\delta:=\frac{d-4}{4(3d+4)}.
\end{align*}
Since $d>4$, we have
$0<\delta<\alpha/d<1/2$ and $\alpha>2$.
The inequality $\alpha/d<1/2$ makes the exponential
term in \eqref{eq:return-good-event} decay, while
$\alpha>2$ ensures that
$\int_1^\infty t^{1-\alpha}\dd t<\infty$.
These choices also give
\begin{align*}
\frac{\alpha}{1+2\delta}
=\frac{3d+4}{8}>\frac d4+1,
\end{align*}
which will be used later.

\begin{lemma}\label{lem:return-before-cutoff}
There is a constant $C_0$, depending only on $d$,
such that, for every $p\ge1$,
\begin{align}\label{eq:return-before-cutoff}
\left(\mathbb E
|m_o\overline p_t(o,o)|^p\right)^{1/p}
\le C_{d,p}t^{-\alpha},
\qquad
1\le t\le t_L,
\end{align}
where
\begin{align*}
t_L:=\left(\frac{L}{16C_0}\right)^{2/(1+2\delta)}.
\end{align*}
The bound holds for either transfer configuration,
uniformly in the positive transfer masses
and all fixed locations.
\end{lemma}

\begin{proof}
Now we choose the parameters in
\eqref{eq:return-good-event}. Choose
\begin{align*}
k=\left\lceil\frac{\alpha}{2\delta}\right\rceil,
\qquad C_0=16k.
\end{align*}
For each $t\ge1$, set
\begin{align*}
\ell=t^\delta,
\qquad
a=t^{\alpha/d},
\qquad
R=C_0t^{1/2+\delta}.
\end{align*}
When $t\le t_L$, these choices satisfy
\begin{align*}
1\le\ell\le\frac{R}{16k},
\qquad
\ell\le a\le\frac L8,
\qquad
R\le\frac L{16}.
\end{align*}
Since $m_o\le1$, \eqref{eq:return-good-event} gives,
on $G(o;R,\ell)$,
\begin{align*}
m_o\overline p_t(o,o)
&\le C_dt^{-\alpha}
+\exp\left(-c_dt^{1-2\alpha/d}\right)
+C_dt^{-2\delta k}\\
&\le C_dt^{-\alpha}.
\end{align*}
Here $1-2\alpha/d>0$ and $2\delta k\ge\alpha$.

For every configuration,
\begin{align*}
0\le m_o\overline p_t(o,o)
\le\mathbb P_o(X_t=o)\le1.
\end{align*}
Splitting according to the good event therefore yields
\begin{align*}
\left(\mathbb E
|m_o\overline p_t(o,o)|^p\right)^{1/p}
&\le C_dt^{-\alpha}
+\mathbb P(G(o;R,\ell)^c)^{1/p}.
\end{align*}
Choose $A$ sufficiently large that
\begin{align*}
\delta A-d(1/2+\delta)\ge p\alpha.
\end{align*}
By \eqref{eq:good-event-probability},
\begin{align*}
\mathbb P(G(o;R,\ell)^c)^{1/p}
&\le C_{d,A,p}(1+R)^{d/p}\ell^{-A/p}\\
&\le C_{d,p}
t^{[d(1/2+\delta)-\delta A]/p}\\
&\le C_{d,p}t^{-\alpha}.
\end{align*}
This proves \eqref{eq:return-before-cutoff}.
\end{proof}

\subsection{Collection of results and variance upper bound}
\label{sec:results_collection3}

Now we collect the results we have generated in this section.

\begin{corollary}\label{cor:return-all-times}
Put
\begin{align*}
\beta:=\frac{\alpha}{1+2\delta}
=\frac{3d+4}{8}>\frac d4+1.
\end{align*}
For every $p\ge1$ and $t\ge0$,
\begin{align}\label{eq:return-all-times}
\left(\mathbb E
|m_o\overline p_t(o,o)|^p\right)^{1/p}
\le C_{d,p}(1+t)^{-\beta}.
\end{align}
The bound holds for either transfer configuration
under the conditioning described in
Section~\ref{sec:small-time}, uniformly in $L$,
the positive transfer masses, and all fixed locations.
\end{corollary}

\begin{proof}
For $L$ sufficiently large, $1\le t_L\le L^2$.
For $0\le t\le1$, use $m_o\overline p_t(o,o)\le1$.
For $1\le t\le t_L$, apply
Lemma~\ref{lem:return-before-cutoff} and $\alpha>\beta$.

By \eqref{eq:return-large-time}, the centered return
density is nonincreasing. Thus, for $t_L\le t\le L^2$,
\begin{align*}
\left(\mathbb E
|m_o\overline p_t(o,o)|^p\right)^{1/p}
&\le C_{d,p}t_L^{-\alpha}\\
&\le C_{d,p}L^{-2\beta}\\
&\le C_{d,p}(1+t)^{-\beta}.
\end{align*}
For $t\ge L^2$, \eqref{eq:return-large-time} gives
\begin{align*}
\left(\mathbb E
|m_o\overline p_t(o,o)|^p\right)^{1/p}
&\le C_{d,p}L^{-2\beta}
e^{-c_d(t-L^2)/L^2}\\
&\le C_{d,p}(1+t)^{-\beta}.
\end{align*}
For bounded $L$, \eqref{eq:return-large-time} and
$m_o\overline p_0(o,o)\le1$ give the same bound
after increasing $C_{d,p}$.
\end{proof}

For one replacement, we now apply
\eqref{eq:response-return-bound}.
For $p\ge2$, Minkowski's inequality and
\eqref{eq:return-all-times} give
\begin{align*}
\left(\mathbb E
\|\psi_s\|_{L^2(\T_L^d)}^p\right)^{2/p}
&=\left\|\|\psi_s\|_{L^2(\T_L^d)}^2
\right\|_{L^{p/2}(\mathbb P)}\\
&\le2\left\|
\sum_{o\in\{p,q\}}\int_0^\infty
t\,\overline p_t(o,o)\dd t
\right\|_{L^{p/2}(\mathbb P)}\\
&\le2\sum_{o\in\{p,q\}}\int_0^\infty
t\,\|\overline p_t(o,o)\|_{L^{p/2}(\mathbb P)}
\dd t\\
&=2\sum_{o\in\{p,q\}}\int_0^\infty t
\left(\mathbb E
|\overline p_t(o,o)|^{p/2}\right)^{2/p}\dd t\\
&\le C_{d,p}
\left(\frac1s+\frac1{1-s}\right)
\int_0^\infty t(1+t)^{-\beta}\dd t.
\end{align*}
Since $\beta>2$, the last integral is finite.
Taking square roots gives the following bound for
$p\ge2$; the case $1\le p<2$ follows from
H\"older's inequality.

\begin{lemma}\label{lem:potential-response-moments}
For every $p\ge1$ and $s\in(0,1)$,
\begin{align}\label{eq:potential-response-moments}
\left(\mathbb E
\|\psi_s\|_{L^2(\T_L^d)}^p\right)^{1/p}
\le C_{d,p}
\left(s^{-1/2}+(1-s)^{-1/2}\right).
\end{align}
The bound is uniform in $L$ and the fixed transfer
locations $\{p,q\}$.
\end{lemma}

As a reward for our hard work in this section,
we prove the asymptotically correct variance upper
bound for the matching cost.

\begin{corollary}\label{cor:variance-upper-bound}
We have
\begin{align}\label{eq:variance-upper-bound}
\Var(C_L)\le C_dn.
\end{align}
\end{corollary}

\begin{proof}
Let $z_1,\ldots,z_{n-1},p,q$
be independent and uniform on $\T_L^d$, and
consider the one-replacement interpolation.
By Lemmas~\ref{lem:walk-response} and
\ref{lem:potential-response-moments},
\begin{align*}
\int_0^1
\left(
\E\sum_{r\in\mathcal Z}m_r(s)|w_r'(s)|^2
\right)^{1/2}\dd s
&=\int_0^1
\left(\E\|\psi_s\|_{L^2}^2\right)^{1/2}\dd s\\
&\le C_d\int_0^1
\left(s^{-1/2}+(1-s)^{-1/2}\right)\dd s\\
&\le C_d.
\end{align*}
Write $\phi_L:=\phi_0$ and
$\Delta\phi_L:=\phi_1-\phi_0$.
Minkowski's inequality gives
\begin{align*}
\left(\E\|\Delta\phi_L\|_{L^2}^2\right)^{1/2}
&\le\int_0^1
\left(\E\|\psi_s\|_{L^2}^2\right)^{1/2}\dd s\\
&\le C_d.
\end{align*}

We next bound the endpoint weights.
Translation invariance and normalization imply
$\E\phi_L(0)=0$.
Efron--Stein and translation invariance therefore give
\begin{align}\label{eq:potential-point-moment}
\E|\phi_L(0)|^2
&\le\frac n2\E|\Delta\phi_L(0)|^2\notag\\
&=\frac12\E\|\Delta\phi_L\|_{L^2}^2
\le C_d.
\end{align}
At $s=0$, every active cell has mass one, so
\begin{align*}
w_r(0)=\int_{C_r(w(0))}
\left(\phi_L(y)+\rho_L(y,r)^2\right)\dd y.
\end{align*}
Let $R_r$ denote its radius.
Jensen's inequality, exchangeability, and
Lemma~\ref{lem:conditional-radius-moments} yield
\begin{align}\label{eq:weight-second-moment}
\E|w_p(0)|^2
&=\frac1n\E
\sum_{r\in\mathcal Z\setminus\{q\}}|w_r(0)|^2
\notag\\
&\le\frac2n\E\|\phi_L\|_{L^2}^2
+\frac2n\E
\sum_{r\in\mathcal Z\setminus\{q\}}R_r^4
\notag\\
&\le2\E|\phi_L(0)|^2+C_d
\le C_d.
\end{align}
The same argument at $s=1$ gives
$\E|w_q(1)|^2\le C_d$.

Lemma~\ref{lem:replacement-response-bound}
now implies
\begin{align*}
\left(\E|F(1)-F(0)|^2\right)^{1/2}
&\le
\left(\E|w_q(1)|^2\right)^{1/2}
+\left(\E|w_p(0)|^2\right)^{1/2}\\
&\quad+\int_0^1
\left(
\E\sum_{r\in\mathcal Z}m_r(s)|w_r'(s)|^2
\right)^{1/2}\dd s\\
&\le C_d.
\end{align*}
Finally, Efron--Stein gives
\begin{align*}
\Var(C_L)
&\le\frac n2\E|F(1)-F(0)|^2\\
&\le C_dn.
\end{align*}
\end{proof}

\section{Replacement estimates}
\label{sec:mixed-replacements}

We use the walk estimates from Section~\ref{sec:geometry}
to control the interaction between two replacements
and the fourth moment of a single replacement cost.
These estimates will be used to verify the conditions
of the martingale central limit theorem.

\subsection{Off-diagonal estimates}
\label{sec:off-diagonal}

We first estimate the transition density between
two distinct sites. For $v\in\mathcal Z$, define
\begin{align*}
\tau_R^v:=\inf\{t\ge0:\rho_L(X_t,v)\ge R\},
\end{align*}
the first exit time of the walk from $B_R(v)$.

\begin{lemma}\label{lem:off-diagonal-exit}
Let $R>0$ and let $o,u\in\mathcal Z$ satisfy
$\rho_L(o,u)\ge2R$. For every $t\ge0$,
\begin{align}\label{eq:off-diagonal-exit}
\sqrt{m_om_u}\,p_{2t}(o,u)
\le
\mathbb P_o(\tau_R^o\le t)^{1/2}
+\mathbb P_u(\tau_R^u\le t)^{1/2}.
\end{align}
\end{lemma}

\begin{proof}
For any $A\subseteq\mathcal Z$, reversibility gives
\begin{align*}
\sum_{r\in A}m_rp_t(o,r)^2
&=\frac1{m_o}\sum_{r\in A}
\mathbb P_o(X_t=r)\mathbb P_r(X_t=o)\\
&\le\frac1{m_o}\mathbb P_o(X_t\in A).
\end{align*}
In particular,
\begin{align*}
\sum_{r\in\mathcal Z}m_rp_t(o,r)^2\le\frac1{m_o}.
\end{align*}
The same bounds hold with $u$ in place of $o$.

Since $B_R(o)\cap B_R(u)=\varnothing$, the semigroup
property, symmetry, and Cauchy--Schwarz give
\begin{align*}
p_{2t}(o,u)
&=\sum_{r\in\mathcal Z}m_rp_t(o,r)p_t(u,r)\\
&\le
\left(\sum_{r\notin B_R(o)}m_rp_t(o,r)^2\right)^{1/2}
\left(\sum_{r\notin B_R(o)}m_rp_t(u,r)^2\right)^{1/2}\\
&\quad+
\left(\sum_{r\in B_R(o)}m_rp_t(o,r)^2\right)^{1/2}
\left(\sum_{r\in B_R(o)}m_rp_t(u,r)^2\right)^{1/2}\\
&\le\frac{
\mathbb P_o(X_t\notin B_R(o))^{1/2}
+\mathbb P_u(X_t\notin B_R(u))^{1/2}
}{\sqrt{m_om_u}}.
\end{align*}
The last inequality uses
$B_R(o)\subseteq B_R(u)^c$.
Being outside a ball at time $t$ requires leaving
it by time $t$, which proves
\eqref{eq:off-diagonal-exit}.
\end{proof}

We now average over the independent uniform background
sites, keeping the transfer and prescribed background
locations fixed, as in Section~\ref{sec:small-time}.
Our goal is to bound
$\int_0^\infty|\overline p_t(o,u)|\dd t$
in terms of the distance $\rho_L(o,u)$. Put
\begin{align}\label{eq:spatial-decay-exponent}
\gamma:=\frac{5d-4}{8}>\frac d2.
\end{align}
Since $2\gamma>d$,
\begin{align*}
\int_{\T_L^d}(1+\rho_L(0,u))^{-2\gamma}\dd u
\le\int_{\R^d}(1+|u|)^{-2\gamma}\dd u
\le C_d.
\end{align*}
This integrability will be used in
Lemma~\ref{lem:replacement-fourth-moment}.

\begin{lemma}\label{lem:off-diagonal-occupation}
Fix either transfer configuration and all its transfer
locations. We may also fix the location of one
background site. Let $o,u$ be two distinct sites
among these fixed sites, and let the remaining
background sites be independent and uniform on
$\T_L^d$. For every $p\ge1$,
\begin{align}\label{eq:off-diagonal-occupation}
\left(
\E\left[
\left(\int_0^\infty
|\overline p_t(o,u)|\dd t\right)^p
\right]
\right)^{1/p}
\le
\frac{C_{d,p}}{\sqrt{m_om_u}}
(1+\rho_L(o,u))^{-\gamma}.
\end{align}
The bound is uniform in $L$, the positive transfer
masses, and the locations of all fixed sites.
\end{lemma}

\begin{proof}
We separate the distance regime $\rho_L(o,u)\le R_0$ and $\rho_L(o,u)> R_0$ for some fixed $R_0$ independent of $L$.
The semigroup property and $\sum_{r\in\mathcal Z}m_r = n$ give
\begin{align*}
\overline p_t(o,u)
=\sum_{r\in\mathcal Z}m_r
\overline p_{t/2}(o,r)\overline p_{t/2}(u,r).
\end{align*}
Cauchy--Schwarz therefore yields
\begin{align*}
|\overline p_t(o,u)|
\le
\sqrt{\overline p_t(o,o)\overline p_t(u,u)}.
\end{align*}
By H\"older's inequality and
Corollary~\ref{cor:return-all-times},
\begin{align}\label{eq:off-diagonal-time-decay}
\left(\E|\overline p_t(o,u)|^p\right)^{1/p}
&\le
\left(
\E\left[
\overline p_t(o,o)^{p/2}
\overline p_t(u,u)^{p/2}
\right]
\right)^{1/p}\notag\\
&\le
\left(\E\overline p_t(o,o)^p\right)^{1/(2p)}
\left(\E\overline p_t(u,u)^p\right)^{1/(2p)}
\notag\\
&\le
\frac{C_{d,p}}{\sqrt{m_om_u}}(1+t)^{-\beta}.
\end{align}
By Minkowski's inequality and $\beta>1$,
\begin{align*}
\left\|\int_0^\infty|\overline p_t(o,u)|\dd t
\right\|_{L^p(\mathbb P)}
&\le\int_0^\infty
\|\overline p_t(o,u)\|_{L^p(\mathbb P)}\dd t\\
&\le\frac{C_{d,p}}{\sqrt{m_om_u}}
\int_0^\infty(1+t)^{-\beta}\dd t\\
&\le\frac{C_{d,p}}{\sqrt{m_om_u}}.
\end{align*}
For $\rho_L(o,u)\le R_0$, the factor
$(1+\rho_L(o,u))^{-\gamma}$ is bounded below by constant
$(1+R_0)^{-\gamma}$, so the preceding bound implies
\eqref{eq:off-diagonal-occupation} after increasing
the constant.

For $\rho_L(o,u)>R_0$, put
\begin{align*}
R:=\frac{\rho_L(o,u)}{32\sqrt d},
\qquad
\theta:=\frac{\gamma}{\beta-1}
=\frac{5d-4}{3d-4}<2,
\qquad
T:=R^\theta.
\end{align*}
Then $R\le L/64$, and the balls $B_R(o)$ and
$B_R(u)$ are disjoint.

Choose an integer $k\ge1$ and $A>0$ such that
\begin{align*}
\frac{k(2-\theta)}2\ge\theta+\gamma,
\qquad
\frac{\theta A/4-d}{p}\ge\theta+\gamma.
\end{align*}
Set $\ell=R^{\theta/4}$. We choose $R_0$ large
enough, depending only on $d,p$, that
\begin{align*}
R\ge1,
\qquad
1\le\ell\le\frac{R}{16k}.
\end{align*}

On $G=G(o;R,\ell)\cap G(u;R,\ell)$,
Lemma~\ref{lem:off-diagonal-exit}, applied at time
$t/2$, and \eqref{eq:walk-exit-iterated} give,
for $0\le t\le T$,
\begin{align*}
\sqrt{m_om_u}\,p_t(o,u)
&\le C_{d,k}
\left(\frac{t+\ell^2}{R^2}\right)^{k/2}\\
&\le C_{d,k}R^{-k(2-\theta)/2}.
\end{align*}
For every configuration, Lemma~\ref{lem:off-diagonal-exit}
gives $\sqrt{m_om_u}\,p_t(o,u)\le2$. By \eqref{eq:good-event-probability} we have
\begin{align}\label{eq:off-diagonal-short-time}
\left\|\sqrt{m_om_u}\,p_t(o,u)\right\|_{L^p(\mathbb P)}
&\le
\left\|\sqrt{m_om_u}\,p_t(o,u)\1_G\right\|_{L^p(\mathbb P)}
+\left\|\sqrt{m_om_u}\,p_t(o,u)\1_{G^c}\right\|_{L^p(\mathbb P)}
\notag\\
&\le C_{d,k}R^{-k(2-\theta)/2}
+2\mathbb P(G^c)^{1/p}\notag\\
&\le C_{d,k}R^{-k(2-\theta)/2}
+2\left[
\mathbb P(G(o;R,\ell)^c)
+\mathbb P(G(u;R,\ell)^c)
\right]^{1/p}\notag\\
&\le C_{d,k}R^{-k(2-\theta)/2}
+C_{d,A,p}(1+R)^{d/p}\ell^{-A/p}\notag\\
&\le C_{d,k}R^{-k(2-\theta)/2}
+C_{d,A,p}R^{(d-\theta A/4)/p}\notag\\
&\le C_{d,p}R^{-(\theta+\gamma)},
\qquad 0\le t\le T.
\end{align}
We now split the time integral at $T$.
Minkowski's inequality,
$|\overline p_t(o,u)|\le p_t(o,u)+1/n$,
and \eqref{eq:off-diagonal-time-decay} give
\begin{align*}
&\left\|\int_0^\infty
|\overline p_t(o,u)|\dd t\right\|_{L^p(\mathbb P)}\\
&\quad\le
\int_0^T\left(
\|p_t(o,u)\|_{L^p(\mathbb P)}+\frac1n
\right)\dd t
+\int_T^\infty
\|\overline p_t(o,u)\|_{L^p(\mathbb P)}\dd t\\
&\quad\le
\frac{C_{d,p}}{\sqrt{m_om_u}}
\left(
TR^{-(\theta+\gamma)}
+\frac{T}{n}
+T^{1-\beta}
\right).
\end{align*}
Here we used $m_o,m_u\le1$ to put the $T/n$ inside the same prefactor.

Since $T=R^\theta$ and $\theta(\beta-1)=\gamma$,
\begin{align*}
TR^{-(\theta+\gamma)}=R^{-\gamma},
\qquad
T^{1-\beta}=R^{-\gamma}.
\end{align*}
Moreover,
\begin{align*}
d-\theta-\gamma
=\frac{(d-4)(9d-4)}{8(3d-4)}>0.
\end{align*}
Thus $n=L^d\ge R^d$ implies
\begin{align*}
\frac{T}{n}\le R^{\theta-d}\le R^{-\gamma}.
\end{align*}
Combining these bounds gives
\begin{align*}
\left\|\int_0^\infty
|\overline p_t(o,u)|\dd t\right\|_{L^p(\mathbb P)}
\le\frac{C_{d,p}}{\sqrt{m_om_u}}R^{-\gamma}.
\end{align*}
Since $R$ is a fixed multiple of
$\rho_L(o,u)$, this proves
\eqref{eq:off-diagonal-occupation}.
\end{proof}

\subsection{Mixed replacement costs}
\label{sec:mixed-costs}

Consider the two-transfer configuration from
Section~\ref{sec:split-paths}.
Write $p_\tau^{s,t}$ for its transition density,
where $\tau$ denotes walk time and $s,t$ are the
transfer parameters.

For fixed $s$, the argument in the proof of
Lemma~\ref{lem:walk-response}, applied to the
second transfer, gives
\begin{align*}
\partial_t w_r(s,t)
=\int_0^\infty
\bigl(
\overline p_\tau^{s,t}(r,q_2)
-\overline p_\tau^{s,t}(r,p_2)
\bigr)\dd\tau.
\end{align*}
Since
$\partial_sF(s,t)=w_{q_1}(s,t)-w_{p_1}(s,t)$,
we obtain
\begin{align}\label{eq:mixed-response-kernel}
\partial_t\partial_sF(s,t)
&=\int_0^\infty
\Bigl[
\overline p_\tau^{s,t}(q_1,q_2)
-\overline p_\tau^{s,t}(q_1,p_2)
\notag\\
&\hspace{35mm}
-\overline p_\tau^{s,t}(p_1,q_2)
+\overline p_\tau^{s,t}(p_1,p_2)
\Bigr]\dd\tau.
\end{align}
Thus the mixed derivative is determined by the
transition densities between the two pairs of
transfer sites.

Fix the four transfer locations and let the background
sites be independent and uniform on $\T_L^d$.
Define the mixed replacement cost by
\begin{align}\label{eq:mixed-replacement-cost}
\Delta_{p_1,q_1}\Delta_{p_2,q_2}C_L
:=F(1,1)-F(1,0)-F(0,1)+F(0,0).
\end{align}

\begin{lemma}\label{lem:mixed-replacement-decay}
For every $p\ge1$,
\begin{align}\label{eq:mixed-replacement-decay}
\left(
\E|\Delta_{p_1,q_1}\Delta_{p_2,q_2}C_L|^p
\right)^{1/p}
\le C_{d,p}
\sum_{\substack{a\in\{p_1,q_1\}\\b\in\{p_2,q_2\}}}
(1+\rho_L(a,b))^{-\gamma}.
\end{align}
The bound is uniform in $L$ and the fixed transfer
locations. It remains valid when one background
location is also fixed, uniformly in that location.
\end{lemma}

\begin{proof}
By \eqref{eq:mixed-response-kernel} and
Lemma~\ref{lem:off-diagonal-occupation},
\begin{align*}
\|\partial_s\partial_tF(s,t)\|_{L^p(\mathbb P)}
&\le
\sum_{\substack{a\in\{p_1,q_1\}\\b\in\{p_2,q_2\}}}
\left\|
\int_0^\infty
|\overline p_\tau^{s,t}(a,b)|\dd\tau
\right\|_{L^p(\mathbb P)}\\
&\le C_{d,p}
\sum_{\substack{a\in\{p_1,q_1\}\\b\in\{p_2,q_2\}}}
\frac{(1+\rho_L(a,b))^{-\gamma}}
{\sqrt{m_a(s)m_b(t)}}.
\end{align*}
Here $m_a(s)$ is either $s$ or $1-s$, and
$m_b(t)$ is either $t$ or $1-t$. Consequently,
for each pair $a,b$,
\begin{align*}
\int_0^1\int_0^1
\frac{\dd t\,\dd s}{\sqrt{m_a(s)m_b(t)}}
&=
\left(\int_0^1\frac{\dd s}{\sqrt{m_a(s)}}\right)
\left(\int_0^1\frac{\dd t}{\sqrt{m_b(t)}}\right)\\
&=4.
\end{align*}
In particular, the preceding estimate with $p=1$
and Tonelli's theorem give
\begin{align*}
\int_0^1\int_0^1
|\partial_s\partial_tF(s,t)|\dd t\,\dd s<\infty.
\end{align*}
By the dominated convergence theorem, we may pass
to the full square in \eqref{eq:mixed-replacement-path},
obtaining
\begin{align}\label{eq:mixed-replacement-integral}
\Delta_{p_1,q_1}\Delta_{p_2,q_2}C_L
=\int_0^1\int_0^1
\partial_s\partial_tF(s,t)\dd t\,\dd s.
\end{align}
Minkowski's inequality now gives
\begin{align*}
\|\Delta_{p_1,q_1}\Delta_{p_2,q_2}C_L\|_{L^p(\mathbb P)}
&\le
\int_0^1\int_0^1
\|\partial_s\partial_tF(s,t)\|_{L^p(\mathbb P)}
\dd t\,\dd s\\
&\le C_{d,p}
\sum_{\substack{a\in\{p_1,q_1\}\\b\in\{p_2,q_2\}}}
(1+\rho_L(a,b))^{-\gamma}.
\end{align*}
The conditional statement follows from the same
argument, since Lemma~\ref{lem:off-diagonal-occupation}
allows one background location to be fixed.
\end{proof}

\subsection{Fourth moment of a replacement cost}
\label{sec:replacement-moments}

Fix distinct locations $p,q$ and let
$\mathcal B=\{z_1,\ldots,z_{n-1}\}$ be the independent
uniform background sites. For the one-transfer path
from Section~\ref{sec:split-paths}, write
\begin{align}\label{eq:single-replacement-cost}
\Delta_{p,q}C_L:=F(1)-F(0).
\end{align}
We next bound its fourth moment. This will control
the fourth moments of the martingale differences
and imply the Lindeberg condition.

\begin{lemma}\label{lem:replacement-fourth-moment}
For fixed distinct locations $p,q$,
\begin{align}\label{eq:replacement-fourth-moment}
\E|\Delta_{p,q}C_L|^4\le C_d.
\end{align}
The bound is uniform in $L,p,q$.
\end{lemma}

\begin{proof}
Translation invariance implies
\begin{align*}
\E[\Delta_{p,q}C_L]=0.
\end{align*}
Let $z_1',\ldots,z_{n-1}'$ be independent uniform
copies of the background sites, independent of
$\mathcal B$, and put
\begin{align*}
J_i:=\Delta_{p,q}\Delta_{z_i,z_i'}C_L.
\end{align*}
Thus $J_i$ is the change in $\Delta_{p,q}C_L$
when $z_i$ is replaced by $z_i'$.
Let $\mathcal G=\sigma(z_1,\ldots,z_{n-1})$.
The fourth-moment Efron--Stein inequality
\cite[Theorem~2]{BBLM05}, followed by conditional
Jensen's inequality and the tower property, gives
\begin{align*}
\E|\Delta_{p,q}C_L|^4
&\le C\E\left[
\left(
\E\left[
\sum_{i=1}^{n-1}|J_i|^2
\,\middle|\,\mathcal G
\right]
\right)^2
\right]\\
&\le C\E\left[
\E\left[
\left(\sum_{i=1}^{n-1}|J_i|^2\right)^2
\,\middle|\,\mathcal G
\right]
\right]\\
&=C\E\left(\sum_{i=1}^{n-1}|J_i|^2\right)^2\\
&=C\sum_{i=1}^{n-1}\E|J_i|^4
+C\sum_{i\ne j}\E[|J_i|^2|J_j|^2].
\end{align*}

For each $i$, write
\begin{align*}
K_i:=
\sum_{\substack{a\in\{p,q\}\\b\in\{z_i,z_i'\}}}
(1+\rho_L(a,b))^{-\gamma}.
\end{align*}
Since $2\gamma>d$, integration over the independent
uniform locations $z_i,z_i'$ gives, for $k=2,4$,
\begin{align*}
\E K_i^k
&\le \frac{C_{d,k}}n
\sum_{a\in\{p,q\}}
\int_{\T_L^d}(1+\rho_L(a,x))^{-k\gamma}\dd x\\
&\le\frac{C_d}{n}.
\end{align*}
By Lemma~\ref{lem:mixed-replacement-decay},
\begin{align*}
\E[|J_i|^4\mid z_i,z_i']
\le C_dK_i^4,
\qquad
\E|J_i|^4\le\frac{C_d}{n}.
\end{align*}

For $i\ne j$, condition on
$z_i,z_i',z_j,z_j'$.
In $J_i$, the site $z_j$ is one fixed background
site, while $z_j'$ is unused. The conditional
statement of Lemma~\ref{lem:mixed-replacement-decay}
therefore gives
\begin{align*}
\E[|J_i|^4\mid z_i,z_i',z_j,z_j']
&\le C_dK_i^4,\\
\E[|J_j|^4\mid z_i,z_i',z_j,z_j']
&\le C_dK_j^4.
\end{align*}
Conditional Cauchy--Schwarz yields
\begin{align*}
&\E[|J_i|^2|J_j|^2\mid z_i,z_i',z_j,z_j']\\
&\qquad\le
\left(
\E[|J_i|^4\mid z_i,z_i',z_j,z_j']
\E[|J_j|^4\mid z_i,z_i',z_j,z_j']
\right)^{1/2}\\
&\qquad\le C_dK_i^2K_j^2.
\end{align*}
The pairs $(z_i,z_i')$ and $(z_j,z_j')$ are
independent, so
\begin{align*}
\E[|J_i|^2|J_j|^2]
\le C_d(\E K_i^2)(\E K_j^2)
\le\frac{C_d}{n^2}.
\end{align*}
Substituting these estimates gives
\begin{align*}
\E|\Delta_{p,q}C_L|^4
\le C_d\left(
\frac{n-1}{n}
+\frac{(n-1)(n-2)}{n^2}
\right)
\le C_d.
\end{align*}
\end{proof}

We also need the fourth-moment bound when one
background location is fixed.

\begin{corollary}\label{cor:conditional-replacement-fourth-moment}
For fixed distinct locations $p,q,a$,
\begin{align}\label{eq:conditional-replacement-fourth-moment}
\E\left[
|\Delta_{p,q}C_L|^4
\,\middle|\,z_1=a
\right]
\le C_d.
\end{align}
The bound is uniform in $L,p,q,a$.
\end{corollary}

\begin{proof}
Keep $z_1=a$ fixed and let $b$ be an independent
uniform site. Let $\Delta'_{p,q}C_L$ denote the
replacement cost after replacing the background
site $a$ by $b$. By the definition of the mixed
replacement cost,
\begin{align*}
\Delta_{p,q}C_L
=
\Delta'_{p,q}C_L
-\Delta_{p,q}\Delta_{a,b}C_L.
\end{align*}
The background sites for $\Delta'_{p,q}C_L$ are
all independent and uniform. Hence
Lemma~\ref{lem:replacement-fourth-moment} gives
\begin{align*}
\E|\Delta'_{p,q}C_L|^4\le C_d.
\end{align*}
Conditional on $b$,
Lemma~\ref{lem:mixed-replacement-decay} gives
\begin{align*}
\left(
\E\left[
|\Delta_{p,q}\Delta_{a,b}C_L|^4
\,\middle|\,b
\right]
\right)^{1/4}
&\le C_d
\sum_{\substack{u\in\{p,q\}\\v\in\{a,b\}}}
(1+\rho_L(u,v))^{-\gamma}\\
&\le C_d.
\end{align*}
Therefore, we have
\begin{align*}
\left(
\E\left[
|\Delta_{p,q}C_L|^4
\,\middle|\,z_1=a
\right]
\right)^{1/4}
&=
\left(
\E\left|
\Delta'_{p,q}C_L
-\Delta_{p,q}\Delta_{a,b}C_L
\right|^4
\right)^{1/4}\\
&\le
\left(\E|\Delta'_{p,q}C_L|^4\right)^{1/4}
+
\left(
\E|\Delta_{p,q}\Delta_{a,b}C_L|^4
\right)^{1/4}\\
&\le C_d.
\end{align*}
Here the expectations on the right are over
$b,z_2,\ldots,z_{n-1}$, with $a,p,q$ fixed.
\end{proof}

\section{Central limit theorem}
\label{sec:clt}

\subsection{Martingale differences}
\label{sec:martingale-differences}

Recall that $X_1,\ldots,X_n$ are independent uniform
sites on $\T_L^d$. Put
\begin{align*}
\mathcal F_0:=\{\varnothing,\Omega\},
\qquad
\mathcal F_i:=\sigma(X_1,\ldots,X_i),
\qquad 1\le i\le n.
\end{align*}
Define the martingale differences by
\begin{align}\label{eq:martingale-differences}
\xi_i
:=\E[C_L\mid\mathcal F_i]
-\E[C_L\mid\mathcal F_{i-1}].
\end{align}
Then
\begin{align*}
C_L-\E C_L=\sum_{i=1}^n\xi_i,
\qquad
\E[\xi_i\mid\mathcal F_{i-1}]=0.
\end{align*}

Let $X_i'$ be an independent uniform site, independent
of the original configuration, and let $C_L^{(i)}$
be the cost after replacing $X_i$ by $X_i'$.
Since
\begin{align*}
\E[C_L^{(i)}\mid\mathcal F_i]
=\E[C_L\mid\mathcal F_{i-1}],
\end{align*}
we have
\begin{align}\label{eq:martingale-replacement}
\xi_i
=\E[C_L-C_L^{(i)}\mid\mathcal F_i]
=-\E[\Delta_{X_i,X_i'}C_L\mid\mathcal F_i].
\end{align}
Conditional Jensen's inequality and
Lemma~\ref{lem:replacement-fourth-moment} give
\begin{align}\label{eq:martingale-fourth-moment}
\E|\xi_i|^4
&\le
\E\left[
\E\left[
|C_L-C_L^{(i)}|^4
\,\middle|\,\mathcal F_i
\right]
\right]\notag\\
&=\E|\Delta_{X_i,X_i'}C_L|^4
\le C_d.
\end{align}

Write
\begin{align}\label{eq:conditional-variance-sum}
Q_L:=\sum_{i=1}^n
\E[\xi_i^2\mid\mathcal F_{i-1}],
\qquad
\sigma_L^2:=\Var(C_L).
\end{align}
The martingale differences are orthogonal, so
\begin{align*}
\E Q_L
=\sum_{i=1}^n\E\xi_i^2
=\sigma_L^2.
\end{align*}
We next show that $Q_L$ concentrates around its
expectation. Together with a lower bound
$\sigma_L^2\ge c_dn$, this will give the conditional
variance convergence required for the martingale
central limit theorem.

\subsection{Concentration of the conditional variance}
\label{sec:conditional-variance}

\begin{lemma}\label{lem:conditional-variance}
We have
\begin{align}\label{eq:conditional-variance}
\Var(Q_L)
\le C_dn^2L^{d-2\gamma}
=C_dn^2L^{-(d-4)/4}.
\end{align}
\end{lemma}

\begin{proof}
Write
\begin{align*}
k_L(x,y):=(1+\rho_L(x,y))^{-\gamma}.
\end{align*}
Since $\gamma<d$ and $\mu_L(\dd x)=\dd x/n$,
\begin{align}\label{eq:averaged-decay-kernel}
\int_{\T_L^d}k_L(x,a)\mu_L(\dd x)
&\le\frac{C_d}{L^d}
\int_0^{\sqrt dL/2}
r^{d-1}(1+r)^{-\gamma}\dd r\notag\\
&\le C_dL^{-\gamma},
\end{align}
uniformly in $a$.

For each $i$, regard the martingale difference as
a function of the newly revealed location:
\begin{align*}
g_i(x)
:=\E\left[
C_L-C_L^{(i)}
\,\middle|\,
X_1,\ldots,X_{i-1},X_i=x
\right].
\end{align*}
For fixed $x$, this is a random variable depending
only on $X_1,\ldots,X_{i-1}$. Since $X_i$ is independent
of these sites and has law $\mu_L$,
\begin{align}\label{eq:conditional-variance-integral}
\xi_i=g_i(X_i),
\qquad
B_i:=\E[\xi_i^2\mid\mathcal F_{i-1}]
=\int_{\T_L^d}g_i(x)^2\mu_L(\dd x).
\end{align}
In particular, $Q_L=\sum_{i=1}^n B_i$.

Fix $j<i$, and replace $X_j$ by an independent
uniform site $X_j'$. Use the superscript $(j)$
for the resulting functions $g_i^{(j)}$ and
variables $B_i^{(j)}$.
Condition on $X_j=a$ and $X_j'=b$.
For each fixed $x$, the remaining randomness in
$g_i(x)$ and $g_i^{(j)}(x)$ comes from the past
sites $(X_r)_{r<i,\ r\ne j}$.

Conditional Jensen's inequality and
Corollary~\ref{cor:conditional-replacement-fourth-moment}
give
\begin{align*}
\E\left[
|g_i(x)|^4
\,\middle|\,X_j=a,\ X_j'=b
\right]
&\le
\int_{\T_L^d}
\E\left[
|\Delta_{x,q}C_L|^4
\,\middle|\,X_j=a
\right]\mu_L(\dd q)
\le C_d,\\
\E\left[
|g_i^{(j)}(x)|^4
\,\middle|\,X_j=a,\ X_j'=b
\right]
&\le
\int_{\T_L^d}
\E\left[
|\Delta_{x,q}C_L|^4
\,\middle|\,X_j=b
\right]\mu_L(\dd q)
\le C_d.
\end{align*}
In each integral, $x,q$ are the transfer locations
and $a$ or $b$ is one fixed background site.
Consequently,
\begin{align*}
\E\left[
|g_i(x)+g_i^{(j)}(x)|^4
\,\middle|\,X_j=a,\ X_j'=b
\right]
\le C_d.
\end{align*}

Using the same future sites and replacement location
$q=X_i'$ in both conditional expectations gives
\begin{align*}
g_i(x)-g_i^{(j)}(x)
=
\int_{\T_L^d}
\E\left[
\Delta_{x,q}\Delta_{a,b}C_L
\,\middle|\,(X_r)_{r<i,\ r\ne j}
\right]\mu_L(\dd q).
\end{align*}
For each fixed $q$, the inner expectation averages
over the future sites $X_{i+1},\ldots,X_n$, with
the transfer locations $x,q,a,b$ and the remaining
past sites fixed.

Minkowski's inequality and conditional Jensen's
inequality give
\begin{align*}
&\left(
\E\left[
|g_i(x)-g_i^{(j)}(x)|^4
\,\middle|\,X_j=a,\ X_j'=b
\right]
\right)^{1/4}\\
&\quad\le
\int_{\T_L^d}
\left(
\E\left|
\E\left[
\Delta_{x,q}\Delta_{a,b}C_L
\,\middle|\,(X_r)_{r<i,\ r\ne j}
\right]
\right|^4
\right)^{1/4}\mu_L(\dd q)\\
&\quad\le
\int_{\T_L^d}
\left(
\E|\Delta_{x,q}\Delta_{a,b}C_L|^4
\right)^{1/4}\mu_L(\dd q)\\
&\quad\le
C_d\int_{\T_L^d}
\bigl[
k_L(x,a)+k_L(x,b)
+k_L(q,a)+k_L(q,b)
\bigr]\mu_L(\dd q)\\
&\quad\le
C_d\bigl[
k_L(x,a)+k_L(x,b)+L^{-\gamma}
\bigr].
\end{align*}
By \eqref{eq:conditional-variance-integral},
\begin{align*}
B_i-B_i^{(j)}
=
\int_{\T_L^d}
(g_i(x)-g_i^{(j)}(x))
(g_i(x)+g_i^{(j)}(x))
\mu_L(\dd x).
\end{align*}
Applying Minkowski's inequality and then
Cauchy--Schwarz under the conditional law gives
\begin{align*}
&\left(
\E\left[
|B_i-B_i^{(j)}|^2
\,\middle|\,X_j=a,\ X_j'=b
\right]
\right)^{1/2}\\
&\quad\le
\int_{\T_L^d}
\left(
\E\left[
|g_i(x)-g_i^{(j)}(x)|^2
|g_i(x)+g_i^{(j)}(x)|^2
\,\middle|\,X_j=a,\ X_j'=b
\right]
\right)^{1/2}
\mu_L(\dd x)\\
&\quad\le
\int_{\T_L^d}
\left(
\E\left[
|g_i(x)-g_i^{(j)}(x)|^4
\,\middle|\,X_j=a,\ X_j'=b
\right]
\right)^{1/4}\\
&\qquad\qquad\times
\left(
\E\left[
|g_i(x)+g_i^{(j)}(x)|^4
\,\middle|\,X_j=a,\ X_j'=b
\right]
\right)^{1/4}
\mu_L(\dd x)\\
&\quad\le
C_d\int_{\T_L^d}
\bigl[
k_L(x,a)+k_L(x,b)+L^{-\gamma}
\bigr]\mu_L(\dd x)\\
&\quad\le C_dL^{-\gamma}.
\end{align*}
The bound is uniform in $a,b$. Squaring and averaging
over $X_j,X_j'$ yields
\begin{align}\label{eq:conditional-variance-replacement}
\E|B_i-B_i^{(j)}|^2
\le C_dL^{-2\gamma},
\qquad j<i.
\end{align}

Let $Q_L^{(j)}$ be obtained from $Q_L$ by replacing
$X_j$ with $X_j'$. Since $B_i$ depends only on
$X_1,\ldots,X_{i-1}$, only terms with $i>j$
can change. Thus
\begin{align*}
Q_L-Q_L^{(j)}
=\sum_{i=j+1}^n(B_i-B_i^{(j)}).
\end{align*}
Cauchy--Schwarz and
\eqref{eq:conditional-variance-replacement} give
\begin{align*}
\E|Q_L-Q_L^{(j)}|^2
&\le(n-j)\sum_{i=j+1}^n
\E|B_i-B_i^{(j)}|^2\\
&\le C_d(n-j)^2L^{-2\gamma}.
\end{align*}
Finally, Efron--Stein gives
\begin{align*}
\Var(Q_L)
&\le\frac12\sum_{j=1}^n
\E|Q_L-Q_L^{(j)}|^2\\
&\le C_dL^{-2\gamma}
\sum_{j=1}^n(n-j)^2\\
&\le C_dn^3L^{-2\gamma}\\
&=C_dn^2L^{d-2\gamma}
=C_dn^2L^{-(d-4)/4}.
\end{align*}
\end{proof}

\subsection{Variance lower bound by Hoeffding decomposition}
\label{sec:variance-lower-bound}

The mixed-replacement estimate also lets us reduce
the variance lower bound to a lower bound for the
effect of a single replacement.

\begin{lemma}\label{lem:variance-lower-reduction}
There is a constant $c_d>0$ such that
\begin{align}\label{eq:variance-lower-reduction}
\Var(C_L)
\ge c_dn
\left(\E|C_L-C_L^{(1)}|^2\right)^2.
\end{align}
\end{lemma}

\begin{proof}
Let $C_L^{(i,j)}$ denote the cost after replacing
both $X_i$ and $X_j$ by independent uniform copies.
Apply the Hoeffding decomposition from
Appendix~\ref{sec:hoeffding}:
\begin{align*}
C_L-\E C_L=\sum_{S\ne\varnothing}U_S,
\qquad
v_S:=\E U_S^2.
\end{align*}
By orthogonality,
\begin{align*}
\Var(C_L)=\sum_{S\ne\varnothing}v_S.
\end{align*}

Let $U_S^{(i)}$ denote the component after replacing
$X_i$ by $X_i'$. If $i\notin S$, then
$U_S^{(i)}=U_S$. If $i\in S$, the centering property
\eqref{eq:hoeffding-centering} gives
\begin{align*}
\E[U_SU_S^{(i)}]=0,
\qquad
\E|U_S-U_S^{(i)}|^2=2v_S.
\end{align*}
Hence, by \eqref{eq:hoeffding-variance}, we have
\begin{align*}
\E|C_L-C_L^{(i)}|^2
&=\E\left|
\sum_{S:\,i\in S}(U_S-U_S^{(i)})
\right|^2\\
&=\sum_{S:\,i\in S}
\E|U_S-U_S^{(i)}|^2\\
&=2\sum_{S:\,i\in S}v_S.
\end{align*}

Similarly, the mixed difference of $U_S$ vanishes
unless $i,j\in S$. When both indices belong to $S$,
the four terms
$U_S,U_S^{(i)},U_S^{(j)},U_S^{(i,j)}$
are pairwise uncorrelated and each has second
moment $v_S$. Orthogonality across distinct sets
therefore gives
\begin{align*}
\E|C_L-C_L^{(i)}-C_L^{(j)}+C_L^{(i,j)}|^2
=4\sum_{S:\,i,j\in S}v_S.
\end{align*}
Summing over indices, we obtain
\begin{align*}
\sum_{S\ne\varnothing}|S|v_S
&=\frac12\sum_{i=1}^n
\E|C_L-C_L^{(i)}|^2,\\
\sum_{S\ne\varnothing}|S|^2v_S
&=\frac12\sum_{i=1}^n
\E|C_L-C_L^{(i)}|^2\\
&\quad+\frac14\sum_{i\ne j}
\E|C_L-C_L^{(i)}-C_L^{(j)}+C_L^{(i,j)}|^2.
\end{align*}
Here the last sum is over ordered pairs.

Lemma~\ref{lem:replacement-fourth-moment} gives
\begin{align*}
\E|C_L-C_L^{(i)}|^2
\le
\left(\E|C_L-C_L^{(i)}|^4\right)^{1/2}
\le C_d.
\end{align*}
For $i\ne j$, condition on
$X_i,X_i',X_j,X_j'$ and apply
Lemma~\ref{lem:mixed-replacement-decay}.
Averaging over these independent uniform locations
and using $2\gamma>d$ yields
\begin{align*}
&\E|C_L-C_L^{(i)}-C_L^{(j)}+C_L^{(i,j)}|^2\\
&\quad\le
C_d\E\sum_{\substack{a\in\{X_i,X_i'\}\\
                     b\in\{X_j,X_j'\}}}
(1+\rho_L(a,b))^{-2\gamma}\\
&\quad\le
\frac{C_d}{n}
\int_{\T_L^d}(1+\rho_L(0,u))^{-2\gamma}\dd u
\le\frac{C_d}{n}.
\end{align*}
Consequently,
\begin{align*}
\sum_{S\ne\varnothing}|S|^2v_S
\le C_dn+C_d\frac{n(n-1)}n
\le C_dn.
\end{align*}

Finally, exchangeability and Cauchy--Schwarz give
\begin{align*}
\left(
\frac n2\E|C_L-C_L^{(1)}|^2
\right)^2
&=
\left(\sum_{S\ne\varnothing}|S|v_S\right)^2\\
&\le
\left(\sum_{S\ne\varnothing}v_S\right)
\left(\sum_{S\ne\varnothing}|S|^2v_S\right)\\
&\le C_dn\,\Var(C_L).
\end{align*}
Rearranging proves
\eqref{eq:variance-lower-reduction}.
\end{proof}

\begin{lemma}\label{lem:replacement-lower-bound}
For all sufficiently large $L$,
\begin{align}\label{eq:replacement-lower-bound}
\E|C_L-C_L^{(1)}|^2\ge c_d>0.
\end{align}
\end{lemma}

\begin{proof}
We first show that two sufficiently close sites
cannot both be close to the barycenters of their cells.

Let $K$ be a convex set of volume one contained
in a ball of radius $R$, and suppose that $K$
lies on one side of a hyperplane $P$.
For every hyperplane $P'$ parallel to $P$, the slice
$K\cap P'$ is contained in a $(d-1)$-dimensional ball
of radius at most $R$. Hence
\begin{align*}
\mathcal H^{d-1}(K\cap P')
\le \omega_{d-1}R^{d-1}
=:A_R,
\end{align*}
where $\omega_{d-1}$ is the volume of the unit ball
in $\R^{d-1}$. A slab of width $t$ next to $P$
therefore contains at most $A_Rt$ of the volume of $K$.
Writing $\overline y_K$ for its barycenter, we obtain
\begin{align*}
\operatorname{dist}(\overline y_K,P)
&=\int_0^\infty
\left|\{y\in K:\operatorname{dist}(y,P)>t\}\right|
\dd t\\
&\ge\int_0^{1/A_R}(1-A_Rt)\dd t
=\frac1{2A_R}
=:\delta_R.
\end{align*}

Consider two lifted cells of radius at most $R$,
with generators $a,b$. The hyperplane on which their weighted costs are equal
separates the two cells. The preceding
bound shows that their barycenters lie at distance
at least $\delta_R$ from this hyperplane on opposite
sides. Thus, if $|a-b|\le\delta_R/2$,
\begin{align*}
2\delta_R
&\le|\overline y_a-\overline y_b|\\
&\le|\overline y_a-a|+|a-b|+|b-\overline y_b|,
\end{align*}
and consequently
\begin{align}\label{eq:good-pair-barycenter-displacement}
\max\bigl\{
|a-\overline y_a|,
|b-\overline y_b|
\bigr\}\ge\delta_R/2.
\end{align}

We next find a positive density of such pairs.
Conditional on two sites $p,q$ at distance at most
one, the probability that all other sites avoid
$B_2(p)\cup B_2(q)$ is
\begin{align*}
\left(
1-\frac{|B_2(p)\cup B_2(q)|}{n}
\right)^{n-2}
\ge c_0>0
\end{align*}
for all sufficiently large $L$.
The constant $c_0$ depends only on $d$.
By \eqref{eq:conditional-radius-moments}, choose
$R$ sufficiently large that
\begin{align*}
\mathbb P\left(
\max\{R_p,R_q\}>R
\,\middle|\,p,q
\right)
\le c_0/2.
\end{align*}
This choice is uniform in the two prescribed locations.

Fix
\begin{align*}
0<\delta<\min\{1,\delta_R/2\}.
\end{align*}
For $i<j$, define the event
\begin{align*}
G_{ij}
:=\bigl\{
&\rho_L(X_i,X_j)<\delta,\quad
R_{X_i}\le R,\quad R_{X_j}\le R,\\
&X_k\notin B_2(X_i)\cup B_2(X_j)
\text{ for every }k\notin\{i,j\}
\bigr\}.
\end{align*}
We call $(X_i,X_j)$ a good pair when $G_{ij}$ occurs.
Let
\begin{align*}
I_{ij}:=\1_{G_{ij}},
\qquad
N_L:=\sum_{i<j}I_{ij}.
\end{align*}
By exchangeability and the preceding probability bounds,
\begin{align*}
\E N_L
&=\sum_{i<j}\E I_{ij}
=\binom n2\mathbb P(I_{12}=1)\\
&\ge
\binom n2\frac{|B_\delta(0)|}{n}\frac{c_0}{2}
\ge c_dn.
\end{align*}

For $L>4R+4$, the cells in each good pair admit
convex lifts of volume one contained in the
radius-$R$ balls around their generators.
By \eqref{eq:good-pair-barycenter-displacement}
and Lemma~\ref{lem:site-derivative}, whenever
$I_{ij}=1$,
\begin{align*}
\max\{|\nabla_{X_i}C_L|,|\nabla_{X_j}C_L|\}
\ge\delta_R.
\end{align*}

No site belongs to two good pairs. Indeed, if
$(X_i,X_j)$ and $(X_i,X_k)$ were both good, then $\rho_L(X_i,X_k)<\delta<2$, whereas $G_{ij}$
requires $X_k\notin B_2(X_i)$. Therefore,
\begin{align*}
N_L
&\le
\sum_{i<j}I_{ij}
\left(
\1_{\{|\nabla_{X_i}C_L|\ge\delta_R\}}
+\1_{\{|\nabla_{X_j}C_L|\ge\delta_R\}}
\right)\\
&\le
\sum_{i=1}^n
\1_{\{|\nabla_{X_i}C_L|\ge\delta_R\}}.
\end{align*}
Taking expectations and using exchangeability gives
\begin{align}\label{eq:site-gradient-lower-probability}
\mathbb P\left(
|\nabla_{X_1}C_L|\ge\delta_R
\right)
&=\frac1n\E\sum_{i=1}^n
\1_{\{|\nabla_{X_i}C_L|\ge\delta_R\}}\notag\\
&\ge\frac{\E N_L}{n}
\ge c_d.
\end{align}

Let $\mathcal B=\{X_2,\ldots,X_n\}$ and let
$f_{\mathcal B}(x)$ be the matching cost with sites
$\mathcal B\cup\{x\}$. Since
$\nabla f_{\mathcal B}(X_1)=\nabla_{X_1}C_L$,
\eqref{eq:site-gradient-lower-probability} yields
\begin{align*}
\mathbb P\left(
|\nabla f_{\mathcal B}(X_1)|\ge\delta_R
\right)\ge c_d.
\end{align*}
Put $a:=\delta_R/\sqrt d$. Whenever
$|\nabla f_{\mathcal B}(X_1)|\ge\delta_R$,
at least one signed coordinate direction
$e\in\{\pm e_1,\ldots,\pm e_d\}$ satisfies
$\nabla f_{\mathcal B}(X_1)\cdot e\ge a$.
By the union bound, we can choose one fixed
such direction $e$ for which
\begin{align*}
\mathbb P\left(
\nabla f_{\mathcal B}(X_1)\cdot e\ge a
\right)
\ge\frac{c_d}{2d}
=:c.
\end{align*}

Set $h:=-(a/2)e$, a fixed displacement opposite
to this direction. On the event
$\{\nabla f_{\mathcal B}(X_1)\cdot e\ge a\}$,
Lemma~\ref{lem:site-derivative}, with mass one,
gives
\begin{align*}
f_{\mathcal B}(X_1+h)
&\le f_{\mathcal B}(X_1)
+\nabla f_{\mathcal B}(X_1)\cdot h+|h|^2\\
&=f_{\mathcal B}(X_1)
-\frac a2\nabla f_{\mathcal B}(X_1)\cdot e
+\frac{a^2}{4}\\
&\le f_{\mathcal B}(X_1)
-\frac{a^2}{2}+\frac{a^2}{4}\\
&=f_{\mathcal B}(X_1)-\frac{a^2}{4}.
\end{align*} Hence
\begin{align*}
\E|f_{\mathcal B}(X_1+h)-f_{\mathcal B}(X_1)|^2
\ge ca^4/16.
\end{align*}

Conditional on $\mathcal B$, both $X_1$ and $X_1+h$
are uniform on the torus. Subtracting the conditional
mean and using $|u-v|^2\le2|u|^2+2|v|^2$ gives
\begin{align*}
\E|f_{\mathcal B}(X_1+h)-f_{\mathcal B}(X_1)|^2
&\le4\E\Var(C_L\mid\mathcal B)\\
&=2\E|C_L-C_L^{(1)}|^2.
\end{align*}
Combining these estimates
proves the lemma.
\end{proof}

Combining Lemmas~\ref{lem:variance-lower-reduction}
and \ref{lem:replacement-lower-bound} gives the following.

\begin{lemma}\label{lem:variance-lower-bound}
For all sufficiently large $L$,
\begin{align}\label{eq:variance-lower-bound}
\sigma_L^2=\Var(C_L)\ge c_dn,
\end{align}
where $c_d>0$ depends only on $d$.
\end{lemma}

\subsection{Proof of Theorem~\ref{thm:clt}}
\label{sec:clt-conclusion}

\begin{proof}[Proof of Theorem~\ref{thm:clt}]
The variance bounds \eqref{eq:variance-order}
follow from Corollary~\ref{cor:variance-upper-bound}
and Lemma~\ref{lem:variance-lower-bound}.

For the central limit theorem, recall that
\begin{align*}
\mathcal F_i&:=\sigma(X_1,\ldots,X_i),\\
\xi_i&:=\E[C_L\mid\mathcal F_i]
       -\E[C_L\mid\mathcal F_{i-1}],\\
Q_L&:=\sum_{i=1}^n
\E[\xi_i^2\mid\mathcal F_{i-1}].
\end{align*}
In particular,
\begin{align*}
C_L-\E C_L=\sum_{i=1}^n\xi_i,
\qquad
\E Q_L=\sigma_L^2.
\end{align*}
By construction, $\xi_i$ is $\mathcal F_i$-measurable
and $\E[\xi_i\mid\mathcal F_{i-1}]=0$.
Its square integrability follows from
\eqref{eq:martingale-fourth-moment}.
Thus $(\xi_i/\sigma_L)_{i=1}^n$ is a
square-integrable martingale difference array.

To apply the martingale central limit theorem
obtained from \cite[Theorem~8.2.4]{Durrett19},
we verify the following two conditions:
\begin{enumerate}
\item
\begin{align}\label{eq:normalized-conditional-variance}
\frac{Q_L}{\sigma_L^2}
\xrightarrow{\mathbb P}1.
\end{align}

\item For every $\varepsilon>0$,
\begin{align}\label{eq:conditional-lindeberg}
\frac1{\sigma_L^2}\sum_{i=1}^n
\E\left[
\xi_i^2\1_{\{|\xi_i|>\varepsilon\sigma_L\}}
\,\middle|\,\mathcal F_{i-1}
\right]
\xrightarrow{\mathbb P}0.
\end{align}
\end{enumerate}
Both limits are taken as $L\to\infty$.

For the first condition,
Lemmas~\ref{lem:conditional-variance}
and \ref{lem:variance-lower-bound} imply
\begin{align*}
\E\left|
\frac{Q_L}{\sigma_L^2}-1
\right|^2
&=\frac{\Var(Q_L)}{\sigma_L^4}\\
&\le C_dL^{-(d-4)/4}
\longrightarrow0.
\end{align*}
This proves
\eqref{eq:normalized-conditional-variance}.

For the second condition, fix $\varepsilon>0$.
Since
\begin{align*}
\xi_i^2\1_{\{|\xi_i|>\varepsilon\sigma_L\}}
\le\frac{|\xi_i|^4}{\varepsilon^2\sigma_L^2},
\end{align*}
the tower property,
\eqref{eq:martingale-fourth-moment}, and the
variance lower bound give
\begin{align*}
&\E\left[
\frac1{\sigma_L^2}\sum_{i=1}^n
\E\left[
\xi_i^2\1_{\{|\xi_i|>\varepsilon\sigma_L\}}
\,\middle|\,\mathcal F_{i-1}
\right]
\right]\\
&\qquad=
\frac1{\sigma_L^2}\sum_{i=1}^n
\E\left[
\xi_i^2\1_{\{|\xi_i|>\varepsilon\sigma_L\}}
\right]\\
&\qquad\le
\frac1{\varepsilon^2\sigma_L^4}
\sum_{i=1}^n\E|\xi_i|^4\\
&\qquad\le
\frac{C_dn}{\varepsilon^2\sigma_L^4}
\le\frac{C_d}{\varepsilon^2n}
\longrightarrow0.
\end{align*}
The quantity inside the outer expectation is
nonnegative, so Markov's inequality proves
\eqref{eq:conditional-lindeberg}.

Both conditions therefore hold, and the
martingale central limit theorem gives
\begin{align*}
\frac{C_L-\E C_L}{\sigma_L}
=\sum_{i=1}^n\frac{\xi_i}{\sigma_L}
\Longrightarrow N(0,1).
\end{align*}
The scaling identity
\eqref{eq:cost-wasserstein-scaling}
gives the corresponding statements for $W_n$.
\end{proof}

\appendix
\section{Uniform radius estimates}\label{app:radius-moments}

We prove radius estimates that remain uniform when
the transfer masses approach zero and when the
locations of finitely many background sites are fixed.

Recall that
\begin{align*}
R_z:=\sup_{y\in C_z}\rho_L(y,z).
\end{align*}
For $o\in\T_L^d$ and $R>0$, put
\begin{align*}
\mathcal K_R(o)
:=\max\left\{
R_z:
z\in B_R(o)
\ \text{or}\
C_z\cap B_R(o)\ne\varnothing
\right\}.
\end{align*}

\begin{lemma}\label{lem:conditional-radius-moments}
Let $d\ge5$ and let $k,j$ be nonnegative integers.
Consider the target measure
\begin{align*}
\nu
&=\sum_{i=1}^{n-k-j}\delta_{X_i}
+\sum_{b=1}^{j}\delta_{a_b}\\
&\quad+\sum_{b=1}^{k}
\bigl((1-s_b)\delta_{p_b}+s_b\delta_{q_b}\bigr),
\qquad 0<s_b<1,
\end{align*}
where $n=L^d\in\mathbb N$ satisfies $n\ge k+j$,
the prescribed locations are distinct, and the
$X_i$ are independent and uniform on $\T_L^d$.

For every $p\ge1$ and every fixed spatial point $o$,
\begin{align}\label{eq:unit-ball-radius-moments}
\E[(1+\mathcal K_1(o))^p]\le C_{d,p,k+j}.
\end{align}
Moreover, for every prescribed site $z$,
\begin{align}\label{eq:conditional-radius-moments}
\E[(1+R_z)^p]\le C_{d,p,k+j}.
\end{align}
The constants depend only on $d$, $p$, and $k+j$,
and are uniform in $L$, $o$, the prescribed
locations, and the positive transfer masses.
\end{lemma}

\begin{proof}
Add $k+j$ independent uniform points to construct
the ordinary binomial measure
\begin{align*}
\nu_0:=\sum_{i=1}^{n}\delta_{X_i}.
\end{align*}
Both measures have mass $n$, and
\begin{align*}
|\nu-\nu_0|(\T_L^d)\le2(k+j),
\end{align*}
where $|\nu-\nu_0|$ denotes the absolute signed
measure.

For a cube $Q$, define
\begin{align*}
\mathcal W_Q(\lambda)
:=W_2^2\left(
\lambda|_Q,\frac{\lambda(Q)}{|Q|}\1_Q\dd y
\right),
\end{align*}
using Euclidean distance and unnormalized measures.
For finite positive measures $\lambda,\lambda'$, we have
\begin{align*}
\mathcal W_Q(\lambda')
\le\mathcal W_Q(\lambda)
+\operatorname{diam}(Q)^2|\lambda'-\lambda|(Q).
\end{align*}
Therefore, fix the center of the cube $Q_a$ as $o$. For the cube $Q_a$ of
side length $a\le L$, we have
\begin{align*}
a^{-d}\mathcal W_{Q_a}(\nu)
\le a^{-d}\mathcal W_{Q_a}(\nu_0)+C_{d,k+j}a^{2-d}.
\end{align*}
By \cite[Theorem~4.5]{GH}, the ordinary binomial
measure has a random scale $r_0\ge1$, with
\begin{align*}
\sup_L \E[r_0^p]\le C_{d,p},
\qquad p\ge1,
\end{align*}
such that
\begin{align*}
a^{-d}\mathcal W_{Q_a}(\nu_0)\le r_0^2
\end{align*}
at every dyadic scale $r_0\le a\le L$.
Since $a^{2-d}\le1$ for $a\ge1$, enlarging $r_0$
gives a random radius $r_*\ge1$ such that
\begin{align*}
\E[r_*^p]&\le C_{d,p,k+j},
\qquad p\ge1,\\
a^{-d}\mathcal W_{Q_a}(\nu)&\le r_*^2
\end{align*}
at every dyadic scale $r_*\le a\le L$.
These are the discrepancy assumptions of
\cite[Theorem~4.1]{GH}.

Let $u,u_0$ be the mean-zero periodic solutions of
\begin{align*}
\Delta u=\nu-1,
\qquad
\Delta u_0=\nu_0-1,
\end{align*}
and let $\eta_a$ be the periodization of
$a^{-d}\eta(\cdot/a)$, where $\eta$ is the smooth
radial averaging kernel used in \cite{GH}. Define
\begin{align*}
h_a(o)&:=\int_{\T_L^d}\eta_a(y-o)\nabla u(y)\dd y,\\
h_a^0(o)&:=\int_{\T_L^d}\eta_a(y-o)\nabla u_0(y)\dd y.
\end{align*}

In the notation of \cite{GH}, $\nu_0=\mu^{1,L}$,
$u_0=u^{1,L}$, and
$h_a^0(o)=\nabla u_a^{1,L}(o)$.
Thus \cite[Lemma~3.1, (3.1)]{GH}, with $r=a$
and $x=o$, gives, for $p\ge2$,
\begin{align*}
\left(\E|h_a^0(o)|^p\right)^{1/p}
&\le C_{d,p}
\left(1+a^{-d(p-2)/(2p)}\right)a^{-(d-2)/2}\\
&\le C_{d,p},
\qquad a\geq 1.
\end{align*}
For $1\le p<2$, the same bound follows from
the second moment.

We compare $h_a$ and $h_a^0$ by Fourier expansion.
Write
\begin{align*}
\widehat\eta(\xi)
&:=\int_{\R^d}\eta(y)e^{-2\pi i\xi\cdot y}\dd y,\\
\widehat{\nu-\nu_0}(\ell)
&:=\int_{\T_L^d}e^{-2\pi i\ell\cdot x/L}
(\nu-\nu_0)(\dd x).
\end{align*}
Since $\Delta(u-u_0)=\nu-\nu_0$, averaging its
gradient against the symmetric kernel $\eta_a$ gives
\begin{align*}
h_a(o)-h_a^0(o)
=-\frac{i}{2\pi L^{d-1}}
\sum_{\ell\in\Z^d\setminus\{0\}}
\frac{\ell}{|\ell|^2}
\widehat\eta(a\ell/L)
\widehat{\nu-\nu_0}(\ell)
e^{2\pi i\ell\cdot o/L}.
\end{align*}
Smoothness and compact support of $\eta$, together
with $|\nu-\nu_0|(\T_L^d)\le2(k+j)$, imply
\begin{align*}
|\widehat\eta(\xi)|
&\le C_d(1+|\xi|)^{-d-1},\\
|\widehat{\nu-\nu_0}(\ell)|
&\le2(k+j).
\end{align*}
Using $a\ge1$, we therefore obtain
\begin{align*}
|h_a(o)-h_a^0(o)|
&\le\frac{C_d(k+j)}{L^{d-1}}
\sum_{\ell\ne0}\frac{|\widehat\eta(a\ell/L)|}{|\ell|}\\
&\le\frac{C_d(k+j)}{L^{d-1}}
\sum_{\ell\ne0}
\frac{(1+|\ell|/L)^{-d-1}}{|\ell|}\\
&\le C_d(k+j).
\end{align*}

Minkowski's inequality now gives
\begin{align*}
\left(\E|h_a(o)|^p\right)^{1/p}
&\le
\left(\E|h_a^0(o)|^p\right)^{1/p}
+C_d(k+j)\\
&\le C_{d,p,k+j}.
\end{align*}

Fix $\epsilon\in(0,1/16)$ and choose $\delta>0$
small enough that $C_d\delta^{2/(d+2)}\le\epsilon$. If $r_*\le\delta a$, then \cite[(4.7)]{GH} gives
\begin{align}\label{eq:appendix-shifted-support}
|x-(y-h_a(o))|
\le C_da\left(\frac{r_*}{a}\right)^{2/(d+2)}
\le\epsilon a
\end{align}
for every pair $(x,y)$ in the support of the lifted
optimal plan such that
\begin{align*}
x\in B_a(o)
\qquad\text{or}\qquad
y\in B_a(o+h_a).
\end{align*}
Here $x$ is a lifted site and $y$ is a source point
allocated to it. If also $|h_a|\le\epsilon a$, then
\begin{align*}
|x-y|
\le |x-(y-h_a)|+|h_a|
\le2\epsilon a.
\end{align*}
Thus, on $\{r_*\le\delta a,\ |h_a|\le\epsilon a\}$,
every such pair has distance at most $2\epsilon a$.

Take $4\le a\le L$ and work on the event
$\{r_*\le\delta a,\ |h_a|\le\epsilon a\}$.
Consider a cell meeting $B_1(o)$, and choose a
lifted source point $y$ in this intersection,
with corresponding lifted site $x$. Since
\begin{align*}
|y-(o+h_a)|\le1+\epsilon a<a,
\end{align*}
we have $y\in B_a(o+h_a)$.
The preceding displacement bound gives
\begin{align*}
|x-o|\le|x-y|+|y-o|
\le2\epsilon a+1<a.
\end{align*}
Since $x\in B_a(o)$,
\eqref{eq:appendix-shifted-support} applies
to every source point allocated to $x$.
The entire cell therefore has radius at most
$2\epsilon a$, and hence
\begin{align*}
\mathcal K_1(o)\le2\epsilon a<a.
\end{align*}

By the moment bounds for $r_*$ and $h_a$,
Markov's inequality gives, for every $q\ge1$,
\begin{align*}
\mathbb P(\mathcal K_1(o)>a)
&\le\mathbb P(r_*>\delta a)
+\mathbb P(|h_a|>\epsilon a)\\
&\le C_{d,q,k+j}a^{-q}.
\end{align*}
Thus, for any $p\ge1$, choosing $q>p$ and integrating
the tail yields
\begin{align*}
\E[(1+\mathcal K_1(o))^p]
&=1+p\int_0^\infty
(1+a)^{p-1}\mathbb P(\mathcal K_1(o)>a)\dd a\\
&\le C_{d,p,k+j},
\end{align*}
proving \eqref{eq:unit-ball-radius-moments}. Taking $o=z$ proves
\eqref{eq:conditional-radius-moments}.

\end{proof}

\section{Hoeffding decomposition}
\label{sec:hoeffding}

Let $X_1,\ldots,X_n$ be independent random variables
and let $Z=f(X_1,\ldots,X_n)$ be square-integrable.
For $S\subseteq\{1,\ldots,n\}$, write
$X_S=(X_i)_{i\in S}$. The Hoeffding decomposition is
\begin{align}\label{eq:hoeffding-decomposition}
Z-\E Z=\sum_{S\ne\varnothing}U_S,
\qquad
U_S:=
\sum_{T\subseteq S}
(-1)^{|S|-|T|}\E[Z\mid X_T],
\end{align}
where $\E[Z\mid X_\varnothing]=\E Z$.

Each $U_S$ depends only on $X_S$ and has mean zero
when any one of its coordinates is averaged out:
\begin{align}\label{eq:hoeffding-centering}
\E[U_S\mid X_{S\setminus\{i\}}]=0,
\qquad i\in S.
\end{align}
The components are orthogonal, so
\begin{align}\label{eq:hoeffding-variance}
\E[U_SU_T]=0\quad(S\ne T),
\qquad
\Var(Z)=\sum_{S\ne\varnothing}\E U_S^2.
\end{align}

\section{Why the product Poincar\'e bound is not sharp}
\label{app:direct-poincare}

The variance upper bound in
Corollary~\ref{cor:variance-upper-bound} does not
follow directly from the product Poincar\'e estimate
used in \cite[Lemma~2.5]{FM26}.
We record the bound obtained by that argument.

By Lemma~\ref{lem:site-derivative},
Cauchy--Schwarz, and the unit cell volumes,
\begin{align*}
\sum_{i=1}^n\E|\nabla_{X_i}C_L|^2
&\le
4\E\sum_{i=1}^n
\int_{C_{X_i}}\rho_L(y,X_i)^2\dd y\\
&=4\E C_L\\
&\le4\E\sum_{i=1}^nR_{X_i}^2
\le C_dn.
\end{align*}
The last inequality follows from
Lemma~\ref{lem:conditional-radius-moments},
conditioning on the location of each site.

The Poincar\'e constant of $\T_L^d$ is
$L^2/(4\pi^2)$, and the same constant applies
to its product. Therefore,
\begin{align}\label{eq:direct-poincare-variance}
\Var(C_L)
&\le\frac{L^2}{4\pi^2}
\sum_{i=1}^n\E|\nabla_{X_i}C_L|^2\notag\\
&\le\frac{L^2}{\pi^2}\E C_L
\le C_dnL^2.
\end{align}
In contrast,
Corollary~\ref{cor:variance-upper-bound} gives
$\Var(C_L)\le C_dn$.

Equivalently, using
\eqref{eq:cost-wasserstein-scaling}, these bounds
for the squared Wasserstein distance are
\begin{align*}
\text{Direct Poincar\'e bound:}\qquad
\Var(W_n)&\le C_dn^{-1-2/d},\\
\text{Bound proved here:}\qquad
\Var(W_n)&\le C_dn^{-1-4/d}.
\end{align*}

\section*{Acknowledgements}
The author thanks Gilles Mordant for helpful
discussions and valuable suggestions that improved the exposition.

\clearpage
\bibliographystyle{alphaurl}
\bibliography{References}

\end{document}